\documentclass[11pt]{article} 
\usepackage{fullpage}
\usepackage{microtype}
\usepackage{graphicx}
\usepackage{subfigure}
\usepackage{booktabs} 

\usepackage{xcolor}      
\usepackage{natbib}      
\usepackage{algorithm}   
\usepackage{algorithmic} 
\usepackage{url}         
\usepackage{eso-pic}     
\usepackage{forloop}

\usepackage{hyperref}

\usepackage{amsmath}
\usepackage{amssymb}
\usepackage{mathtools}
\usepackage{amsthm}

\usepackage[capitalize,noabbrev]{cleveref}

\theoremstyle{plain}
\newtheorem{theorem}{Theorem}[section]

\newtheorem{lemma}[theorem]{Lemma}

\theoremstyle{definition}

\newtheorem{assumption}[theorem]{Assumption}
\theoremstyle{remark}
\newtheorem{remark}[theorem]{Remark}

\title{\textbf{Revisiting Stochastic Proximal Point Methods:\\ Generalized Smoothness and Similarity}}

\author{Zhirayr Tovmasyan \quad\  Grigory Malinovsky \quad\  Laurent Condat \quad\  Peter Richt\'{a}rik\\[4mm]
Computer Science Program, CEMSE Division, \\
King Abdullah University of Science and Technology (KAUST),\\
Thuwal, 
Kingdom of Saudi Arabia}

 \date{October 3, 2025. Authors' final version. \\To appear in \emph{Journal of Nonlinear and Variational Analysis}, 2026}

\begin{document}

\maketitle

\begin{abstract}
The growing prevalence of nonsmooth optimization problems in machine learning has spurred significant interest in generalized smoothness assumptions. Among these, the $(L_0, L_1)$-smoothness assumption has emerged as one of the most prominent. While proximal methods are well-suited and effective for nonsmooth problems in deterministic settings, their stochastic counterparts remain underexplored. This work focuses on the stochastic proximal point method (SPPM), valued for its stability and minimal hyperparameter tuning—advantages often missing in stochastic gradient descent (SGD). We propose a novel $\phi$-smoothness framework and provide a comprehensive analysis of SPPM without relying on traditional smoothness assumptions. Our results are highly general, encompassing existing findings as special cases. Furthermore, we examine SPPM under the widely adopted expected similarity assumption, thereby extending its applicability to a broader range of scenarios. Our theoretical contributions are illustrated and validated by practical experiments.
\end{abstract}

{
\renewcommand\baselinestretch{0}
\tableofcontents
\renewcommand\baselinestretch{1}
}

\section{Introduction}

In this paper, we address the \textbf{stochastic optimization} problem of minimizing the expected function
\begin{equation}
\label{eq:stoch_opt}
    \min _{x \in \mathbb{R}^d}\big\{f(x):=\mathbb{E}_{\xi \sim \mathcal{D}}\left[f_{\xi}(x)\right]\big\},
\end{equation}
where $\xi \sim \mathcal{D}$ is a random variable drawn from the distribution $\mathcal{D}$, and $\mathbb{E}[\cdot]$ represents the mathematical expectation. Here, $x$ represents a machine learning (ML) model with $d$ parameters, $\mathcal{D}$ denotes the distribution of labeled examples, $\xi \sim \mathcal{D}$ are the samples, $f_{\xi}$ represents the loss associated with a single sample $\xi$, and $f$ corresponds to the generalization error. In this setting, while an unbiased estimator of the gradient $\nabla f_{\xi}(x)$ can be computed, the gradient $\nabla f(x)$ itself is not directly accessible. Such problems form the backbone of supervised learning theory \citep{bottou2018optimization, sun2019survey, sun2020optimization}.

A particular case of interest is the finite-sum optimization problem, where $f$ is the average of a large number of functions \citep{agarwal2015lower, schmidt2017minimizing}:
\begin{equation}
	\label{finite sum}
   \min _{x \in \mathbb{R}^d}\left\{f(x):=\frac{1}{n} \sum_{i=1}^n f_i(x)\right\}.   
\end{equation}  
This problem frequently arises when training supervised ML models using empirical risk minimization \citep{feldman2016generalization, gambella2021optimization}. 
It is an instance of \eqref{eq:stoch_opt} with \(\mathcal{D}\) the uniform distribution over the finite set $\{1,\ldots,n\}$.

Optimization problems in ML and Deep Learning (DL) are frequently \textbf{nonsmooth}, meaning that the gradient of the objective function does not necessarily satisfy the Lipschitz continuity condition, or is even not well defined \citep{iutzeler2020nonsmoothness, gorbunov2024methods}. For instance, even in relatively simple neural networks, the gradient of the standard \(\ell_2\)-regression loss fails to satisfy Lipschitz continuity \citep{zhang2019gradient}.  
Moreover, the convex but nonsmooth setting often provides an effective framework for capturing the complexities of DL problems \citep{khaled2023dowg}. Many widely used and empirically successful adaptive optimization algorithms, such as AdaGrad \citep{duchi2011adaptive}, Adam \citep{kingma2014adam}, and Prodigy \citep{mishchenko2023prodigy} have been specifically designed for this setting, demonstrating their practical effectiveness across various DL applications.  

In contrast, optimization methods that rely on smoothness assumptions and offer strong theoretical guarantees frequently fall short in practical DL tasks \citep{crawshaw2022robustness}. For example, while variance-reduced methods \citep{reddi2016stochastic, gower2020variance} achieve superior convergence rates in theory, they are often outperformed by simpler methods in practice due to the challenges posed by the complex and highly nonconvex landscapes of DL \citep{defazio2019ineffectiveness}. These challenges have motivated researchers to introduce more realistic smoothness assumptions and develop corresponding theoretical guarantees within these refined frameworks \citep{vankov2024optimizing, gorbunov2024methods}.  

One of the earliest extensions beyond the standard Lipschitz smoothness assumption is the \textbf{$\boldsymbol{(L_0, L_1)}$-smoothness} condition, which was initially proposed for twice-differentiable functions \citep{zhang2019gradient}. This assumption posits that the norm of the Hessian can be bounded linearly by the norm of the gradient. Later, this assumption was generalized to encompass a broader class of differentiable functions \citep{zhang2020improved, chen2023generalized}.  

Stochastic Gradient Descent (SGD) methods have been extensively analyzed in both convex and nonconvex settings, with significant attention also given to adaptive variants and other modifications \citep{zhao2021convergence, faw2023beyond, wang2023convergence, hubler2024parameter}. A natural extension of the $(L_0, L_1)$-smoothness assumption involves bounding the Hessian norm with a polynomial dependence on the gradient norm, offering a more flexible and generalized formulation. A further and even more general approach to smoothness involves the use of an arbitrary nondecreasing continuous function to bound the Hessian norm \citep{li2024convex}. This generalized setting not only encompasses the previously discussed assumptions but also provides a broader and more adaptable framework applicable to a wide range of functions \citep{tyurin2024toward}.  

While SGD methods have been extensively studied in the context of generalized smoothness, \textbf{stochastic proximal point methods (SPPM)} remain relatively underexplored. SPPM can serve as an effective alternative when stochastic proximity operators are computationally feasible \citep{asi2019stochastic, asi2020minibatch, khaled2022faster}. We recall that the proximity operator (\textbf{prox}) of a function $f: \mathbb{R}^d \rightarrow \mathbb{R} \cup \{ +\infty \}$ is 
\begin{equation}
\mathrm{prox}_{f}(x) \stackrel{\text{def}}{=} \underset{z \in \mathbb{R}^d}{\mathrm{argmin}}\left\{ f(z) + \frac{1}{2} \|z - x\|^2 \right\}\label{eqprox}
\end{equation}
\citep{bauschke2017correction}. 
Proximal methods, which leverage proxs of functions \citep{par14,con23}, are known for their robustness and resilience to the choice of stepsize, often allowing for the use of larger stepsizes compared to standard gradient-based methods. \citet{ryu2014stochastic} provide convergence rate guarantees for SPPM and emphasize its stability with respect to inaccuracies in learning rate selection; a property not typically observed in SGD. \citet{asi2019stochastic} investigate a more general framework, AProx, which encompasses SPPM as a special case. They establish both stability and convergence rate results for AProx under convexity assumptions. Moreover, SPPM has been shown to achieve convergence rates comparable to those of SGD across a variety of algorithmic settings \citep{richtarik2024unified}.

Alternatively, instead of relying on the smoothness assumption and its generalizations, we can consider the \textbf{similarity} assumption \citep{shamir2014communication}. It reflects the idea that there is a certain level of similarity or homogeneity among the gradients, which is particularly relevant in ML,  where these gradients capture the characteristics of the underlying data \citep{goodfellow2016deep, sun2022distributed}. Recognizing this natural property, several recent works have explored various formulations and generalizations of the similarity assumption \citep{hendrikx2020statistically, szlendak2021permutation, sadiev2024stochastic}. In particular, multiple studies have analyzed stochastic proximal methods under different similarity conditions, demonstrating their practical relevance and theoretical significance \citep{gasanov2024speeding, sadiev2024stochastic}. Moreover, the concept of similarity offers a more refined perspective on the behavior of optimization algorithms in distributed and federated learning settings \citep{karimireddy2020scaffold, kovalev2022optimal, jiang2024stabilized}.

To make the considered methods practical, the computation of a prox 
often involves an \textbf{inexact} approach, where it is approximated by several iterations of a subroutine designed to solve the corresponding auxiliary problem. This technique has been extensively studied in the literature, with various criteria established to ensure the effectiveness of the approximation \citep{kovalev2022optimal,grudzien2023can, borodich2023optimal}. These criteria typically include conditions such as relative and absolute error thresholds \citep{khaled2022faster, li2024convergence}, as well as guarantees on the reduction of the gradient norm \citep{sadiev2022communication}. Meeting these criteria is crucial to maintaining the overall convergence properties and efficiency of the optimization algorithm while balancing computational cost and accuracy.

Our \textbf{contributions} are the following.
\begin{itemize}
    \item {\bf New generalized smoothness assumption, called \textbf{$\boldsymbol{\phi}$-smoothness}.} We investigate more general conditions under which SPPM converges and introduce the novel notion of  $\phi$-smoothness (\Cref{sec:phi-smoothness}). Under this assumption, we establish rigorous convergence guarantees and explore various special cases, highlighting the specific effects and implications of the proposed framework.

\item {\bf Convergence under $\phi$-smoothness.} We conduct a comprehensive analysis of SPPM when the prox is computed inexactly under the newly introduced $\phi$-smoothness assumption (\Cref{sec:gen-smooth}). Specifically, we derive conditions on the number of subroutine steps required to solve the auxiliary problem, ensuring that the overall iteration complexity remains the same as in the case of exact proximal evaluations. Our results provide practical guidelines for balancing computational efficiency and theoretical guarantees.  
Our convergence analysis covers both strongly convex and general convex settings, and we precisely characterize the convergence rates,
offering insights into the tradeoffs between problem complexity and algorithmic performance.  

\item {\bf Convergence under expected similarity.} We further extend our theoretical contributions to settings with the expected similarity assumption (\Cref{sec:similar}), which captures practical scenarios where the different functions share a certain degree of similarity in expectation. Under this assumption, we derive specific convergence results, offering valuable theoretical insights.

\item {\bf Experiments.} To support our theoretical findings, we conduct a series of carefully designed experiments that empirically validate our predictions and provide deeper insights into the practical performance of the proposed methods (\Cref{sec:experiments}).   
\end{itemize}

\section{Preliminaries}

Before presenting our convergence results, we define
the key concepts and outline the assumptions used throughout the work. We start with the standard assumptions that apply to all our results, primarily that each stochastic function \( f_\xi \) is convex and differentiable (for simplicity, to avoid complicated notations associated to set-valued subdifferentials, see their use in \citet{sadiev2024stochastic})

\begin{assumption}[Differentiability] \label{assumption:differentiability}
The function \( f_\xi : \mathbb{R}^d \to \mathbb{R} \) is differentiable for $\mathcal{D}$-almost every  sample $\xi$.
\end{assumption}

We implicitly assume that differentiation and expectation can be interchanged, which leads to \(\nabla f(x) = \nabla \mathbb{E}_{\xi \sim \mathcal{D}}\left[f_{\xi}(x)\right] = \mathbb{E}_{\xi \sim \mathcal{D}}\left[\nabla f_{\xi}(x)\right]\). Consequently, 
$f$ is differentiable.

\begin{assumption}[Convexity] \label{assumption:convexity}
 The function \( f_\xi : \mathbb{R}^d \to \mathbb{R} \) is convex for $\mathcal{D}$-almost every  sample $\xi$; that is,
\begin{equation*}
   f_\xi(x)\geq f_\xi(y) +\langle \nabla f_\xi(y),x-y\rangle \quad \forall x,y \in \mathbb{R}^d.
\end{equation*}
\end{assumption}
Next, we formulate the strong convexity assumption, which is applied to the expected function $f$ in several settings:
\begin{assumption}[Strong convexity of $f$]\label{assumption:strong_convexity}
The function $f:\mathbb{R}^d\to \mathbb{R}$ is $\mu$-strongly convex for some $\mu > 0$; that is,
\begin{equation*}
   f(x)\geq f(y) +\langle \nabla f(y),x-y\rangle + \frac{\mu}{2}\|x-y\|^2 \quad  \forall x,y \in \mathbb{R}^d.
\end{equation*}
\end{assumption}
A strongly convex function is better behaved and easier to minimize than a merely convex one.

Next, we consider the interpolation condition, where all stochastic functions share a common minimizer. This regime is commonly observed in overparameterized models, particularly large-scale ML models, where the number of parameters exceeds the number of training examples, enabling perfect fit to the data \citep{ma2018power, varre2021last}.

\begin{assumption}[Interpolation regime] \label{assumption:interpolation}
 There exists \( x_\star \in \mathbb{R}^d \) such that \( \nabla f_\xi(x_\star) = 0 \)  for $\mathcal{D}$-almost every  sample $\xi$.
\end{assumption}

Now we need to formulate several smoothness assumptions which are used in the analysis. We start with  the well-known standard  assumption of $L$-smoothness: 
\begin{assumption}[$L$-smoothness]\label{L-smoothness}
    The function $f_{\xi} \colon \mathbb{R}^d \to \mathbb{R}$ is $L$-smooth for $\mathcal{D}$-almost every sample $\xi$; that is, it is differentiable and its gradient is  $L$-Lipschitz continuous:
    \begin{equation*}
        \|\nabla f_{\xi}(x) - \nabla f_{\xi}(y)\| \leq L \|x - y\|\quad  \forall x,y \in \mathbb{R}^d.
    \end{equation*}
\end{assumption}
Next, we introduce the \((L_0, L_1)\)-smoothness assumption, initially formulated 
by \citet{zhang2020improved}.

\begin{assumption}[Symmetric \((L_0, L_1)\)-smoothness ]\label{L0L1-smoothness}
    The function $f_{\xi} \colon \mathbb{R}^d \to \mathbb{R}$ is symmetrically $(L_0, L_1)$-smooth for $\mathcal{D}$-almost every sample $\xi$; that is, 
    \begin{align*}
        \|\nabla f_{\xi}(x) - \nabla f_{\xi}(y)\|\leq \left(L_0+L_1 \sup _{u \in[x, y]}\|\nabla f_{\xi}(u)\|\right) \|x - y\|\quad  \forall x,y \in \mathbb{R}^d,
    \end{align*}
   where 
   $[x, y]$ denotes the line segment connecting points \( x \) and \( y \).
\end{assumption}

A more generalized form of symmetric \((L_0, L_1)\)-smoothness is the so-called \(\alpha\)-symmetric generalized smoothness, introduced by \citet{chen2023generalized}. This extension provides a more flexible framework for analyzing smoothness properties by incorporating an additional parameter \(\alpha\), which allows for a broader class of functions to be captured within the smoothness definition. 
   
   \begin{assumption}[\(\alpha\)-symmetric generalized smoothness]\label{alpha_L0L1-smoothness}
   	The function $f_{\xi} \colon \mathbb{R}^d \to \mathbb{R}$ is $\alpha$-symmetrically generalized-smooth for $\mathcal{D}$-almost every sample $\xi$; that is, 
   	\begin{align*}
   		 \|\nabla f_{\xi}(x) - \nabla f_{\xi}(y)\|\leq \left(L_0+L_1 \sup _{u \in[x, y]}\|\nabla f_{\xi}(u)\|^{\alpha}\right) \|x - y\|\quad  \forall x,y \in \mathbb{R}^d,
   	\end{align*}
   	where $\alpha \in [0,1]$ and 
   	$[x, y]$ is 	the interval connecting the points \( x \) and \( y \).
   \end{assumption}
It is important to note that when \(\alpha = 0\), the definition reduces to the standard \(L\)-smoothness, and when \(\alpha = 1\), it corresponds to the symmetric \((L_0, L_1)\)-smoothness.

Then we consider the \textbf{expected similarity} assumption. Initially, the concept of similarity was introduced for twice-differentiable functions and is commonly referred to as Hessian similarity \citep{tian2022acceleration}. For functions that are only once differentiable, analogous assumptions have been employed in the literature. In this work, we consider an even more general assumption introduced by \citet{sadiev2024stochastic}. We call it \emph{Star Similarity} since it only has to hold  with respect to a solution \( x_\star \).
\begin{assumption}[Star Similarity]
\label{assm:similarity}
   The function $f_{\xi} \colon \mathbb{R}^d \to \mathbb{R}$ has Star Similarity; 
   that is,   there exist a solution $x_{\star}$ of \eqref{eq:stoch_opt} and a constant $\delta_{\star}>0$ such that 
\begin{align*}
 \mathbb{E}_{\xi \sim \mathcal{D}} \left[ \left\| \nabla f_\xi(x) - \nabla f(x) - \nabla f_\xi(x_\star) \right\|^2 \right] 
\leq \delta_\star^2 \|x - x_\star\|^2 \quad \forall x \in \mathbb{R}^d. 
\end{align*}
\end{assumption}
Finally, we consider the inexact computation of the prox by employing a subroutine. To ensure convergence guarantees, we impose an assumption on the convergence rate of the subroutine, which has been previously introduced in \citet{sadiev2022communication, gasnikov2022optimal}.
\begin{assumption}[Inexact Proximal Condition]
\label{assumption:inexactness_condition}
At the \( k \)-th iteration of the algorithm, the subroutine \( \mathcal{M} \), after \( T \) iterations, produces an approximate solution \( \hat{x}_k \) that satisfies the  inexactness condition
\begin{equation*}
\|\nabla \Psi_k(\hat{x}_k)\|^2 \leq \frac{\eta \|x_k - x_k^\Psi\|^2}{T^\alpha}\label{eq:inexact_condition},
\end{equation*}
where \( \eta > 0 \) is a scaling factor, \( \alpha > 0 \) denotes the decay rate, and
 \( x_k^\Psi \) denotes the exact minimizer of the proximal subproblem objective $\Psi_k$; that is,
\begin{align}
	x_k^\Psi := \arg \min_{x \in \mathbb{R}^d} \left\{ \Psi_k(x) := f_{\xi_k}(x) + \frac{1}{2\gamma} \|x - x_k\|^2 \right\}.\label{eqpsi}
\end{align}
\end{assumption}

\subsection{The Stochastic Proximal Point Method}

The prox of a function $f$, defined in \eqref{eqprox}, is a well-defined operator when $f$  is proper, closed and convex. 
Moreover, when $f$ is differentiable, it satisfies 
\begin{equation}
y = \mathrm{prox}_{\gamma f}(x) \Leftrightarrow y + \gamma \nabla f(y) = x.\label{eqoptc}
\end{equation}
This relationship provides an insightful connection between proximal operators and gradient-based methods, further motivating their use in optimization frameworks. 

\begin{algorithm}[t]
\caption{Stochastic Proximal Point Method (SPPM)}\label{alg:SPPM}
\begin{algorithmic}[1]
\STATE \textbf{Parameters:} stepsize $\gamma >0$, starting point $x_0 \in \mathbb{R}^d$

\FOR{$t = 0, 1, \ldots, T-1$}
    \STATE Sample $\xi_k \sim \mathcal{D}$
    \STATE $
x_{k+1} := \underset{z \in \mathbb{R}^d}{\mathrm{argmin}}\left\{ f_{\xi_k}(z) + \frac{1}{2\gamma} \|z - x_k\|^2 \right\}
$
\ENDFOR
\end{algorithmic}
\end{algorithm}

In real world applications it is often not possible or feasible to compute the prox of the total objective function $f$, and a stochastic approach comes in handy. The stochastic proximal point method (SPPM), shown as Algorithm \ref{alg:SPPM},  consists in applying at each iteration the prox of a randomly chosen $f_{\xi}$:
\begin{equation}
    x_{k+1}:=\mathrm{prox}_{\gamma f_{\xi}}\left(x_k\right),
\end{equation}
so that for differentiable functions we have 
\begin{equation}
\label{eq:prox_gd_step_}
    x_{k+1}=x_k-\gamma \nabla f_{\xi}\left(x_{k+1}\right).
\end{equation}
Compared to the standard SGD update, SPPM appears very similar, with one small  but conceptually 
major distinction: the gradient is evaluated at the updated point \( x_{k+1} \). Proximal algorithms, such as SPPM, are 
thus implicit methods, as their updates require solving an equation in which the new iterate appears on both sides.\medskip

The main  notations used throughout the paper are summarized in Table \ref{tab1}.
\begin{table}[t]
	\centering
	\caption{Summary of the main notations used in the paper.}\label{tab1}
	\vskip 0.15in
	\begin{tabular}{|l|p{10cm}|}
		\hline
		\textbf{Symbol} & \textbf{Description} \\ \hline
		\( x_* \) & optimal solution that minimizes \( f(x) \) \\
		\hline
		\( x_0 \) &  initial point \\
		\hline
		\( x_k \) &  \( k \)-th iterate of SPPM \\ \hline
		\( \gamma \) &  stepsize of SPPM \\ \hline
		\( \phi(\cdot, \cdot) \) & smoothness function. \\ \hline
		\( D_{f_{\xi}}(x, y) \) &  Bregman divergence of \( f_{\xi} \) between \( x \) and \( y \), see \eqref{eqbreg} \\ \hline
		\( \mathcal{F}_k \) &  \( \sigma \)-algebra generated by \( x_0, \ldots, x_k \) \\ \hline
		\( \hat{x}_k \) & uniformly chosen iterate from the set \( \{x_0, x_1, \ldots, x_{k-1}\} \) \\ \hline
		\( x_k^\Psi \) & exact solution to the proximal subproblem, see \eqref{eqpsi} \\ 
		\hline
		\( T\) & number of inner iterations 
		for  inexact computation of the prox \\ 
		\hline
		\( \sigma_*^2 \) & upper bound on the variance of the stochastic gradients \\ \hline
		\( \mu \) &  strong convexity parameter of the function \( f \) \\ \hline
		\( \delta_* \) &  Star Similarity constant \\ \hline
	\end{tabular}
\end{table}

\section{Related work}
Stochastic Gradient Descent (SGD) \citep{robbins1951stochastic} is a fundamental and widely used optimization algorithm for training machine learning models. Due to its efficiency and scalability, it has become the backbone of modern deep learning, with state-of-the-art training methods relying on various adaptations of SGD \citep{zhou2020towards, wu2020noisy, sun2020optimization}. Over the years, the algorithm has been extensively studied, leading to a deeper understanding of its convergence properties, robustness, and efficiency in different settings \citep{bottou2018optimization, gower2019sgd, khaled2020better, demidovich2023guide}. This ongoing theoretical research continues to refine SGD and its variants, ensuring their effectiveness in large-scale and complex learning tasks \citep{carratino2018learning, nguyen2018sgd, yang2024two}. Notably, methods designed to leverage the smoothness of the objective often struggle in Deep Learning, where optimization problems are inherently nonsmooth. For instance, while variance-reduced methods \citep{johnson2013accelerating, defazio2014saga, schmidt2017minimizing, nguyen2017sarah, malinovsky2023random, cai2023cyclic} theoretically offer faster convergence for finite sums of smooth functions, they are often outperformed in practice by standard, non-variance-reduced methods \citep{defazio2019ineffectiveness}. These challenges highlight the need to explore alternative assumptions that go beyond the standard smoothness assumption.

One of the most commonly used generalized smoothness assumptions is \((L_0, L_1)\)-smoothness \citep{zhang2019gradient}. Several studies have analyzed SGD methods under this condition in the convex setting \citep{koloskova2023revisiting, takezawa2024polyak, li2024convex, gorbunov2024methods, vankov2024optimizing, lobanov2024linear}. The analysis of SGD in the non-convex case was first discussed in \citep{zhang2019gradient} and later extended to momentum-based methods in \citep{zhang2020improved}.  

Similar results have been established for various optimization methods, including Normalized GD \citep{zhao2021convergence, hubler2024gradient, chen2023generalized}, SignGD \citep{crawshaw2022robustness}, AdaGrad-Norm/AdaGrad \citep{faw2023beyond, wang2023convergence}, Adam \citep{wang2024provable}, and Normalized GD with Momentum \citep{hubler2024parameter}. Additionally, methods specifically designed for distributed optimization have been analyzed under generalized smoothness conditions \citep{crawshaw2024federated, demidovich2024methods, khirirat2024error}.

Stochastic proximal point methods \citep{bertsekas2011incremental} have been extensively studied across different settings due to their versatility and strong theoretical properties. This framework can encompass various optimization algorithms, making it a unifying approach for analyzing and designing new methods. One of its key advantages is enhanced stability, which helps mitigate the challenges of variance in stochastic optimization. Additionally, it is particularly well-suited for nonsmooth problems, where traditional smoothness-based methods may struggle. \citep{bianchi2016ergodic, toulis2016towards, davis2019stochastic} These properties make stochastic proximal point methods a valuable tool in both theoretical analysis and practical applications \citep{davis2017three, condat2022distributed, condat2024simple}.

Stochastic proximal methods have become increasingly important in Federated Learning due to their ability to handle decentralized optimization problems efficiently \citep{konevcny2016federated}. Some researchers propose replacing the standard local update steps with the proximity operator, which provides a more robust framework for understanding the behavior of local methods and can lead to faster convergence rates by improving the optimization process \citep{li2020federated, t2020personalized, jhunjhunwala2023fedexp, grudzien2023improving, li2024power}. On the other hand, other studies focus on interpreting the aggregation step as a proximity operator, which allows for a more efficient combination of local updates \citep{mishchenko2022proxskip, condat2022randprox, malinovsky2022variance, condat2023tamuna, hu2023tighter, jiang2024federated}.

\section{\(\phi\)-Smoothness: A New Generalized Assumption} \label{sec:phi-smoothness}

The standard analysis of gradient-based methods, including proximal point methods, can be viewed within the framework of gradient-type algorithms due to their update formulation, as given by \eqref{eq:prox_gd_step_}. While such analyses typically assume standard \( L \)-smoothness, they, in fact, operate under a path-wise smoothness condition, which applies specifically to the sequence of points generated by the algorithm:  
\begin{equation}
\label{eq:path-smooth}
    \left\| \nabla f_\xi(x_{k+1}) - \nabla f_{\xi}(x_k) \right\| \leq L_{\text{path}} \|x_{k+1} - x_k\|.
\end{equation}  
This condition is less restrictive than the conventional \( L \)-smoothness assumption (Assumption~\ref{L-smoothness}), as it only needs to hold for the iterates produced by the algorithm. However, verifying this condition is considerably more challenging, as it depends on the algorithm’s trajectory and specific properties of the method. To circumvent this difficulty, we propose an alternative approach by establishing path-wise smoothness through a different set of assumptions. Specifically, we introduce the generalized \( \phi \)-smoothness assumption, which, when combined with additional conditions, facilitates the derivation of the desired path-wise smoothness property. 

Now we are ready to formulate our novel generalized smoothness assumption:
\begin{assumption}[$\phi$-smoothness]
\label{phi-smoothness}
    The function $f_{\xi} \colon \mathbb{R}^d \to \mathbb{R}$ is $\phi$-smooth for $\mathcal{D}$-almost every sample $\xi$; that is, 
        \begin{equation*}
        \|\nabla f_{\xi}(x) - \nabla f_{\xi}(y)\| \leq \phi\big(\|x - y\|, \|\nabla f_{\xi}(y)\|\big) \|x - y\|\quad \forall x,y \in \mathbb{R}^d,
            \end{equation*}
    where $\phi$ is a nonnegative and nondecreasing function in both variables.
\end{assumption}
This assumption is similar in nature to the concept of \( l \)-smoothness, as discussed in \citet{li2024convex, tyurin2024toward} However, \( l \)-smoothness specifically applies to twice differentiable functions. Next, we present an important lemma for the new class of functions, which plays a crucial role in the analysis.
\begin{lemma}
\label{def:phi_bregman}
  $\phi$-smoothness (Assumption~\ref{phi-smoothness}) implies that, for $\mathcal{D}$-almost every sample $\xi$,
    \begin{align*}
        f_{\xi}(x) \leq f_{\xi}(y) + \langle\nabla f_{\xi}(y),x-y\rangle + \frac{\phi\big(\|x-y\|,\|\nabla f_\xi(y)\|\big)}{2}\|x-y\|^2\quad \forall x,y \in \mathbb{R}^d.
    \end{align*}
\end{lemma}
Next, we demonstrate that the notion of \( \phi \)-smoothness encompasses \((L_0, L_1)\)-smoothness as a special case.
\begin{lemma}
        \label{lemma:connection}
If the function \( f_{\xi} \colon \mathbb{R}^d \to \mathbb{R} \) satisfies the \((L_0, L_1)\)-smoothness condition, then it is also \(\phi\)-smooth with the function \(\phi\) defined as 
\begin{align*}
    \phi\left(\|x-y\|, \|\nabla f_{\xi}(y) \|\right)
     = \left( L_0 + L_1 \|\nabla f_{\xi}(y) \| \right) \exp\left(L_1\|x-y\|\right).
\end{align*}
\end{lemma}
Another special case of the $\phi$-smoothness assumption is 

\begin{lemma}
	\label{lemma:connection1}
	If the function \( f_{\xi} \colon \mathbb{R}^d \to \mathbb{R} \) satisfies the $\alpha$-symmetric generalized-smoothness condition, then it is also \(\phi\)-smooth with the function \(\phi\) defined as 
	\begin{align*}
		\phi\left(\|x-y\|, \|\nabla f_{\xi}(y) \|\right)= K_0+K_1\|\nabla f_\xi(y)\|^\alpha+K_2\left\|x-y\right\|^{\frac{\alpha}{1-\alpha}},
	\end{align*}
where $K_0:=L_0\left(2^{\frac{\alpha^2}{1-\alpha}}+1\right), K_1:=L_1 \cdot 2^{\frac{\alpha^2}{1-\alpha}} \cdot 3^\alpha, K_2:=L_1^{\frac{1}{1-\alpha}} \cdot 2^{\frac{\alpha^2}{1-\alpha}} \cdot 3^\alpha(1-\alpha)^{\frac{\alpha}{1-\alpha}}.$
\end{lemma}

\section{Convergence Results under $\phi$-Smoothness} \label{sec:gen-smooth}

\begin{algorithm}[t]
	\caption{SPPM-inexact} 
	\label{alg:SPPM_inexact}
	\begin{algorithmic}[1]
		\STATE \textbf{Parameters:} stepsize $\gamma >0$, starting point $x_0 \in \mathbb{R}^d$
		
		\FOR{$t = 0, 1, \ldots, T-1$}
		\STATE Sample $\xi_k \sim \mathcal{D}$
		\STATE $
		\hat{x}_{k}:\approx\underset{z \in \mathbb{R}^d}{\mathrm{argmin}}\left\{ f_{\xi_k}(z) + \frac{1}{2\gamma} \|z - x_k\|^2 \right\}
		$
		\STATE $x_{k+1} := x_k - \gamma \nabla f_{\xi_k}(\hat{x}_k)$
		\ENDFOR
	\end{algorithmic}
\end{algorithm}
First, we outline some  intuitions and key ideas behind our results. Our analysis crucially depends on the fact that the distance between two consecutive iterates \(x_k\) and \(x_{k+1}\) is bounded. This result is established using the interpolation regime, convexity of the objective function, and nonexpansiveness of the prox.

We begin by proving monotonicity of the distance to the solution, i.e., \(\|x_{k+1} - x_*\|^2 \leq \|x_k - x_*\|^2\). Building upon this fundamental property, we derive the following key results:
\begin{lemma}
\label{eq:iterates_monotonicity}
    Let Assumptions \ref{assumption:convexity} and \ref{assumption:interpolation} hold. Then the iterates of SPPM satisfy, for every $k\geq 0$,  
        \begin{equation*}
        \|x_k-x_{k+1}\|^2 \leq \|x_0 - x_*\|^2. 
    \end{equation*}
\end{lemma}
Next, we utilize the bound on the distance \(\|x_k - x_{k+1}\|^2\) between consecutive iterates in conjunction with the \(\phi\)-smoothness assumption (Assumption \ref{phi-smoothness}) to establish the concept of path-wise smoothness \eqref{eq:path-smooth}. By deriving several intermediate steps, we further obtain a bound on the difference between the function values evaluated at sequential iterates. 
\begin{lemma}\label{lemma42}
    Let Assumptions \ref{assumption:differentiability}, \ref{assumption:convexity} and \ref{phi-smoothness} hold. Then for any $\gamma>0$,  the iterates of SPPM
     satisfy, for every $k\geq 0$,
    \begin{align}
    f_{\xi_k}(x_k) - f_{\xi_k}(x_{k+1}) \leq \left(\frac{1}{\gamma} + \frac{ \phi\left(\|x_0 - x_{*}\|, \frac{1}{\gamma} \|x_0 - x_{*}\|\right)}
    {2}\right) \|x_k - x_{k+1}\|^2.\label{eq:iterates_lower_bound}
\end{align}
\end{lemma}
We are now ready to present the final result and the main theorem. We establish convergence results, notably without imposing a bound on the stepsize. This implies that while increasing the stepsize can reduce the number of iterations, it also makes each iteration more computationally complex. We provide convergence results for both the strongly convex case and the general convex case. In the strongly convex setting, we achieve a linear rate of convergence, while in the general convex case, we obtain a sublinear rate of $\mathcal{O}(\frac{1}{k})$.

\begin{theorem}\label{theo43}
Let Assumptions \ref{assumption:differentiability} (Differentiability), \ref{assumption:convexity} (Convexity), \ref{assumption:interpolation} (Interpolation) and \ref{phi-smoothness} ($\phi$-smoothness) hold.   Then for any stepsize $\gamma>0$, 
the  iterates of SPPM 
satisfy, for every $k\geq 0$,
\begin{align*}
\mathbb{E}[f(\hat{x}_k)] - f(x_*)\leq  \frac{\phi(\|x_0-x_*\|,\frac{1}{\gamma}\|x_0-x_*\|) + \frac{2}{\gamma}}{2k}\|x_0 - x_*\|^2,
\end{align*}
where $\hat{x}_k$ is a vector chosen from the collection of iterates $x_0$, \ldots, $x_{k-1}$ uniformly at random.

If in addition Assumption \ref{assumption:strong_convexity} holds, then for any stepsize $\gamma>0$, the iterates of SPPM 
satisfy, for every $k\geq 0$, 
\begin{align*}
\mathbb{E} \left[ \| x_{k} - x_* \|^2 \right] 
\leq \left( 1 - \frac{\mu}{\frac{2}{\gamma} + \phi(\|x_0-x_*\|,\frac{1}{\gamma}\|x_0-x_*\|)}\right)^k \| x_0 - x_* \|^2.
\end{align*}
\end{theorem}
It is worth mentioning that our result guarantees convergence to the exact solution, as the variance in the interpolation regime does not hinder convergence to the optimal point. Additionally, it is important to note that there is an additional dependence on the distance between the starting point and the solution, \(\|x_0 - x_*\|^2\), within the function \(\phi\).

Next, we analyze the inexact computation of the prox within the framework of $\phi$-smoothness. In this setting, the prox is computed approximately, followed by an additional gradient step performed from the approximate point:
\begin{align*}
&\hat{x}_{k}:\approx\underset{z \in \mathbb{R}^d}{\mathrm{argmin}}\left\{ f(z) + \frac{1}{2\gamma} \|z - x_k\|^2 \right\},\\
&x_{k+1} := x_k - \gamma \nabla f(\hat{x}_k).
\end{align*}
It is worth noting that if the prox is computed exactly, the additional gradient step becomes redundant. In this case, \(\nabla f_{\xi}(\hat{x}_k) = \frac{1}{\gamma}(x_k - \hat{x}_k)\), and we have \(x_{k+1} = \mathrm{prox}_{\gamma f_{\xi}}(x_k)\).  SPPM with inexact prox (SPPM-inexact) is shown as Algorithm \ref{alg:SPPM_inexact}.

Remarkably, we show that when the number of iterations of the subroutine used to solve the inner subproblem is sufficiently large (satisfying a specific condition), the overall convergence guarantees are preserved up to a constant factor.

\begin{theorem}\label{theo44}
	Let Assumptions \ref{assumption:differentiability} (Differentiability), \ref{assumption:convexity} (Convexity), \ref{assumption:interpolation} (Interpolation) and \ref{phi-smoothness} ($\phi$-smoothness) hold. Consider 
SPPM-inexact with every inexact prox satisfying Assumption \ref{assumption:inexactness_condition}. If \( T \) is chosen sufficiently large such that
	$\frac{\eta \gamma^2}{T^\alpha}\leq c < 1$, then, for any stepsize \( \gamma > 0 \) the iterates of 
	SPPM-inexact
	satisfy, for every $k\geq 0$,
	\begin{align*}
	\mathbb{E}[f(\hat{x}_k)] - f(x_*) \leq \frac{\phi(\|x_0-x_*\|,\frac{1}{\gamma}\|x_0-x_*\|) + \frac{2}{\gamma}}{2k\left(1-c\right)}\|x_0 - x_*\|^2.
	\end{align*}
	If in addition Assumption \ref{assumption:strong_convexity} holds, then for any \( \gamma > 0 \), the iterates of 
SPPM-inexact satisfy, for every $k\geq 0$,
	\begin{align*}
		\mathbb{E}\big[\|x_k - x_*\|^2\big] \leq \left( 1 - \frac{(1-c)\mu}{\frac{2}{\gamma} + \phi(\|x_0-x_*\|,\frac{1}{\gamma}\|x_0-x_*\|)}\right)^{k}\|x_{0} - x_*\|^2.
	\end{align*}
\end{theorem}

\begin{remark}[Realizability of Assumption~\ref{assumption:inexactness_condition} by a parameter-free inner solver]
Although $\phi$-smoothness does not imply global Lipschitz continuity, the proximal subproblem
$\Psi_k$ is $\tfrac{1}{\gamma}$–strongly convex and has a \emph{locally} Lipschitz gradient on a
bounded set containing the iterates (independent of $k$). Consequently, gradient descent on
$\Psi_k$ with Armijo backtracking (requiring no knowledge of $L_0$, $L_1$, or explicit values of
$\phi$) converges linearly. In particular, the inexactness model in
Assumption~\ref{assumption:inexactness_condition} is realizable by such a parameter-free inner
method. For more details, see Appendix~\ref{sec:practical_details}.
\end{remark}

\section{Convergence Results under Expected Similarity} \label{sec:similar}

In this section, we discuss the convergence results under expected similarity (Assumption \ref{assm:similarity}) without assuming that all stochastic gradients vanish at the solution. The main concept of the proof is to introduce the average iterate:
\begin{equation*}
\bar{x}_{k+1} = \mathbb{E}[x_{k+1}|\mathcal{F}_k],
\end{equation*}
where we denote by \(\mathcal{F}_k\) the \(\sigma\)-algebra generated by the randomness (e.g., stochastic gradients or iterates) up to iteration \(k\). This approach enables us to derive the complexity of the algorithm. Since we utilize the expected similarity assumption, we do not require any form of smoothness.

In this setting, convergence is impacted by the variance, defined as follows:
\begin{assumption}[Bounded Variance at Optimum]\label{assumption:bounded_variance}
	Let \( x_* \) denote any minimizer of \( f \), supposed to exist. The variance of the stochastic gradients at \( x_* \) is bounded as:
	\begin{equation*}
		\mathbb{E}_{\xi \sim \mathcal{D}}\big[\|\nabla f_\xi(x_*)\|^2\big] \leq \sigma^2_{*}.
	\end{equation*}
\end{assumption}

\begin{remark}
	The interpolation regime, where stochastic gradients vanish at the solution, is a special case of the general setting with \(\sigma_*^2 = 0\). Detailed convergence results for this regime are provided in Appendix~\ref{app:interpolation}.
\end{remark}

First, we present the convergence result for the exact formulation of SPPM. 
\begin{theorem}\label{theo63}
	Let Assumptions \ref{assumption:differentiability} (Differentiability), \ref{assumption:convexity} (Convexity), \ref{assumption:strong_convexity} (Strong convexity of \(f\)), \ref{assm:similarity} (Star Similarity), and \ref{assumption:bounded_variance} hold. If the stepsize satisfies \(\gamma \leq \frac{\mu}{4\delta_*^2}\), then SPPM satisfies, for every \(k\geq 0\),
	\begin{equation*}
	\mathbb{E}\big[\|x_k - x_*\|^2\big] \leq \left(1 - \frac{1}{2}\min\left(\frac{\gamma\mu}{2},1\right)\right)^{k}\|x_0-x_*\|^2+ 4\max\left(\frac{2}{\gamma\mu},1\right)\gamma^2\sigma_*^2.
	\end{equation*}
\end{theorem}
The result shows linear convergence to a neighborhood of the solution, with the neighborhood size proportional to the variance \(\sigma_*^2\) from Assumption \ref{assumption:bounded_variance}, and depends on the stepsize \(\gamma\). Using a decaying stepsize schedule leads to sublinear convergence to the exact solution.

We now present the convergence guarantee for the inexact variant of SPPM.
\begin{theorem}\label{theo64}
	Let Assumptions \ref{assumption:differentiability} (Differentiability), \ref{assumption:convexity} (Convexity), \ref{assm:similarity} (Star Similarity), and \ref{assumption:bounded_variance} (Bounded Variance) hold. Consider SPPM-inexact with Assumption~\ref{assumption:inexactness_condition} satisfied. If the stepsize satisfies \(\gamma \leq \frac{\mu(1-c)}{4\delta_*^2}\) and \(T\) is chosen sufficiently large such that
	\(\frac{\eta \gamma^2}{T^\alpha} \leq c\), where \(0<c<1\) is a constant, then SPPM-inexact satisfies, for every \(k\geq 0\),
	\begin{equation*}
	\mathbb{E}\big[\|x_k - x_*\|^2\big] \leq \left(1 - \frac{1}{2}\min\left(\frac{\gamma\mu}{2},1-c\right)\right)^{k}\|x_0-x_*\|^2 + \max\left(\frac{2}{\gamma\mu},\frac{1}{1-c}\right)\frac{4\gamma^2\sigma_*^2}{(1-c)}.
	\end{equation*}
\end{theorem}
The result guarantees convergence to a neighborhood of the solution, with a neighborhood size strictly larger than in the exact case due to the approximation error $c$. Specifically, the radius of this neighborhood increases by a factor of at least $\frac{1}{1-c}$. Additionally, the contraction factor becomes worse (closer to $1$), and the allowable stepsize is reduced by a factor of $(1-c)$. Hence, the convergence speed is slower compared to the exact case, but convergence remains guaranteed as long as the approximation error is controlled.

\begin{figure*}[t]
	\centering
	\includegraphics[width=0.48\textwidth]{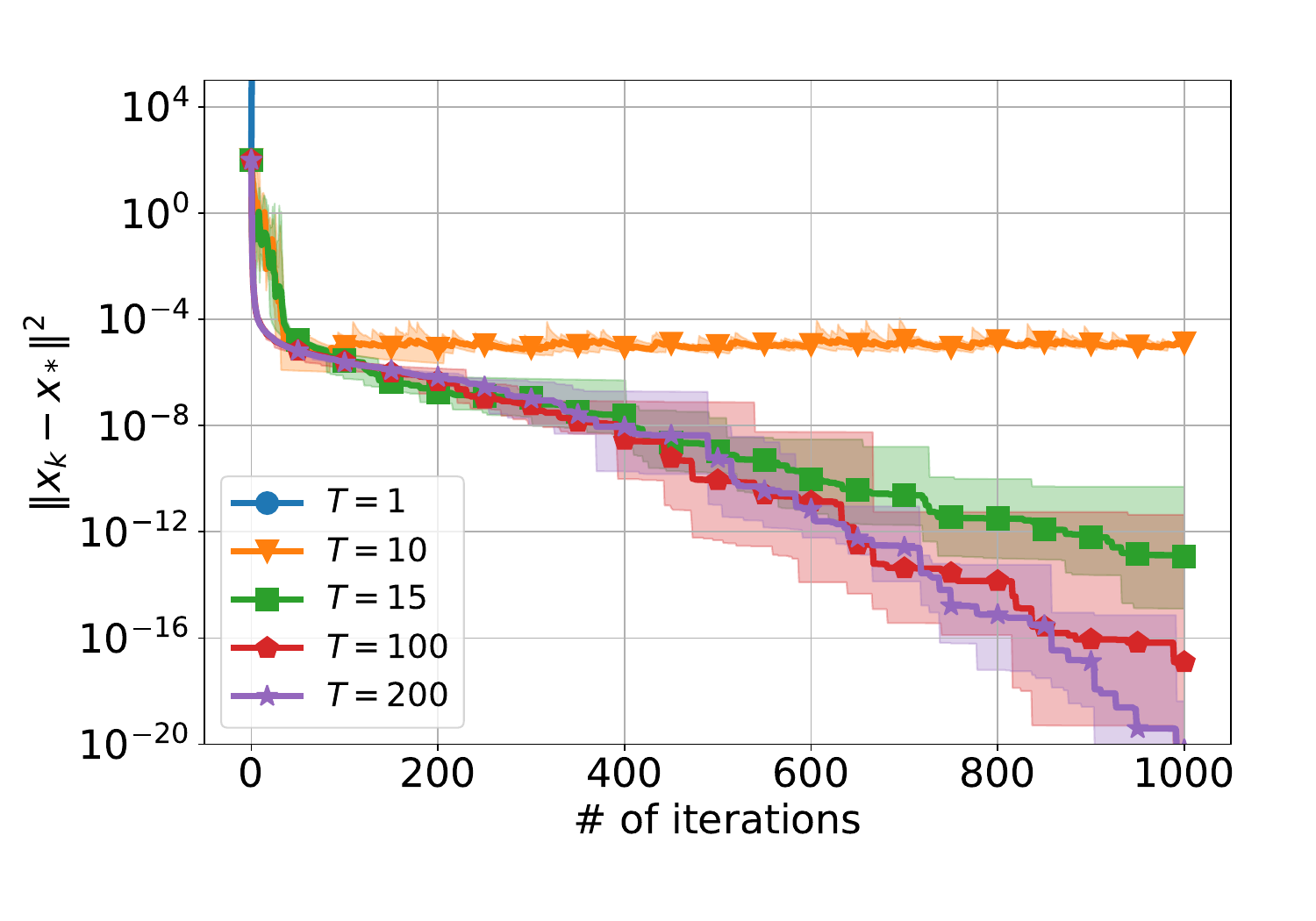}\quad 
	\includegraphics[width=0.48\textwidth]{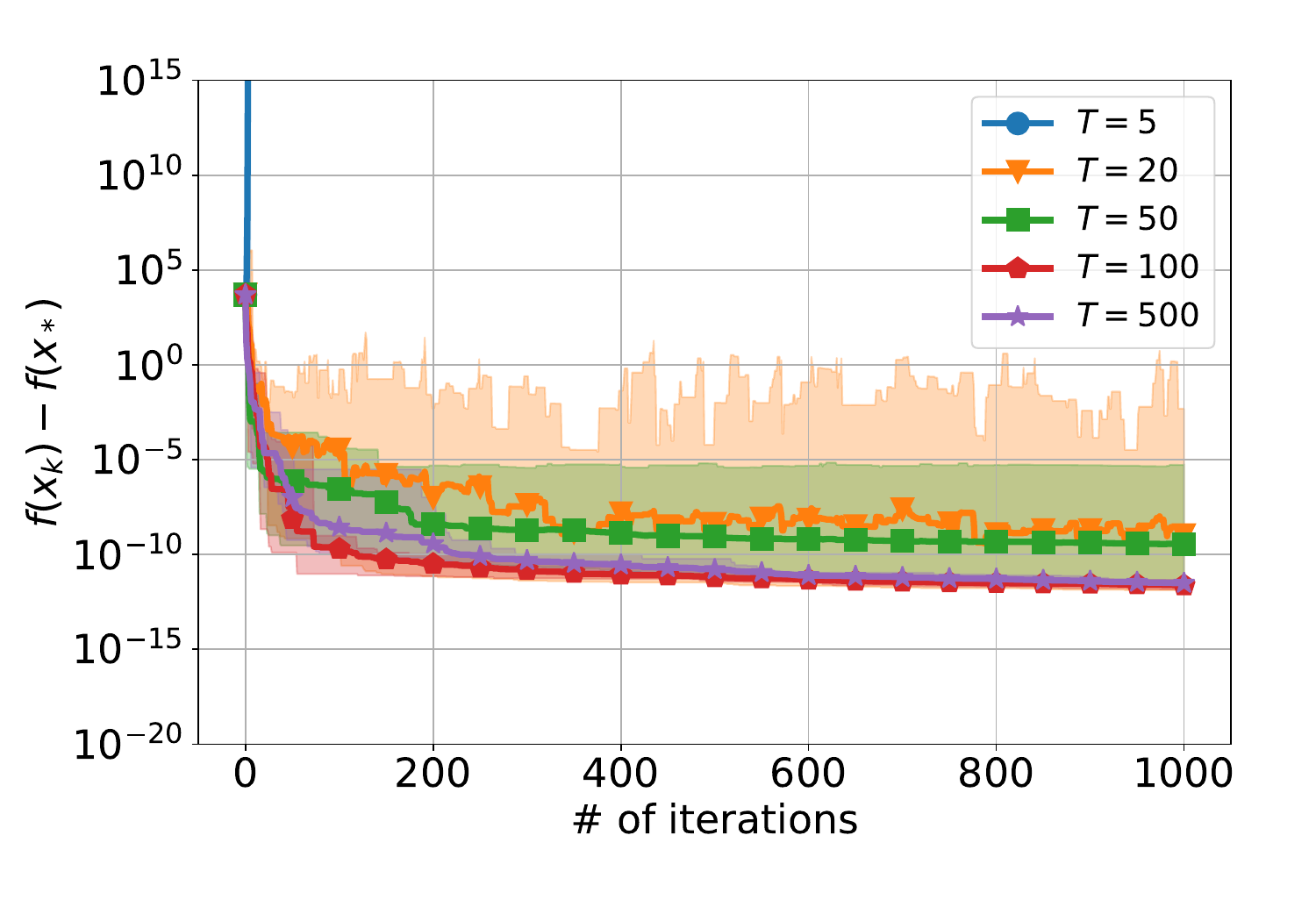}
	\vspace{-0.25cm}
	\caption{Convergence behavior of SPPM-inexact with different inner iterations in strongly convex and convex settings.}\label{fig:linereg3}
\end{figure*}

\section{Experiments} \label{sec:experiments}

In this section, we present numerical experiments conducted for the optimization problem in the finite-sum form \eqref{finite sum}, where the functions \( f_i \) take the specific form 
$
f_i(x) = a_i \|x\|^{2s},
$
with \( s \in \mathbb{N} \setminus\{1\} \) and \( a_i \in \mathbb{R}^{+} \) for all \( i = 1, \dots, n \). Each function \( f_i \) is \((L_0, L_1)\)-smooth, with \( L_0 = 2s \) and \( L_1 = 2s - 1 \), and is convex.

\subsection{Main Experiments}

We begin with two main experiments. The first experiment investigates the effect of varying the maximum number of iterations for the inner solver, without enforcing a stopping condition based on the gradient norm. The purpose of this analysis is to demonstrate that if the inner solver fails to achieve sufficient accuracy in solving the proximal step, the method may either diverge or exhibit slower convergence. We conduct two variations of this experiment. 

In the first variation, we consider functions with different values of \(s\), ranging from \(2\) to \(20\), and run the algorithm with varying maximum iteration limits for the inner solver: \(T = 5, 20, 50, 100, 500\). The number of functions is set to \(n = 100\), and the dimension is \(d = 5\). For the second variation, we modify the original problem to ensure strong convexity of \(f\), enabling us to verify that the convergence rate becomes linear in this setting. The modified problem \eqref{finite sum} is defined as follows:
$
f_i(x) = a_i \|x\|^{2s} + \lambda \langle e_i, x \rangle^2,
$
where \(a_i \in \mathbb{R}^{+}\) for all \(i = 1, \dots, n\), \(\lambda \in \mathbb{R}^{+}\), and \(e_i\) is a unit vector with its \(i\)-th coordinate equal to one. In this case, each function \(f_i\) is \((L_0, L_1)\)-smooth, with \(L_0 = 2s + 2\lambda\) and \(L_1 = 2s - 1\), and convex. Furthermore, \(f\) is \(\frac{\lambda}{n}\)-strongly convex. We set \(n = d = 100\), \(\lambda = 2\), and \(s = 2\), and conduct experiments with inner solver iteration limits of \(1, 10, 15, 100, 200\).

As shown in Figure~\ref{fig:linereg3}, if the number of iterations of the subroutine is sufficiently small, the method may either diverge or exhibit significantly slower convergence. Once the number of inner iterations reaches a sufficiently large value, the convergence improves. It is worth noting that increasing the number of iterations beyond this point does not result in a significant further improvement in convergence. These observations confirm our theoretical findings.

In the second experiment, we analyze the performance of SPPM-inexact and compare it with stochastic gradient descent (SGD) using constant stepsizes. Each function in this experiment has the form
$f_i(x) = a_i |x|^{\frac{2-\alpha_i}{1-\alpha_i}}$, each of which is individually \(\alpha_i\)-smooth, as shown by \citet{chen2023generalized}. It is straightforward to verify that each \( f_i \) is also \(1>\beta>\alpha_i>0\)-symmetric smooth, implying all functions are \(\alpha = \max_{i=1\dots N} \alpha_i\)-symmetric smooth. We set \(N=1000\), with coefficients \(a_i\) uniformly distributed in the range \([10,1000)\). 

We run both SPPM-inexact and SGD with three constant stepsizes: \(\gamma \in \{10^{-3}, 5\times10^{-7}, 10^{-8}\}\). Figure~\ref{fig:comparison_sppm_sgd} illustrates that for smaller stepsizes (\(\gamma=5\times10^{-7}, 10^{-8}\)), both methods exhibit slow convergence, stagnating at suboptimal solutions within the number of iterations considered. However, for the larger stepsize (\(\gamma=10^{-3}\)), SPPM-inexact converges consistently to the optimal solution, whereas SGD diverges.

This experiment highlights an advantage of SPPM-inexact: its capability to utilize relatively larger constant stepsizes effectively, resulting in faster and stable convergence. On the other hand, constant stepsize SGD may require smaller stepsizes or careful tuning to achieve convergence. Hence, while the considered optimization scenario is not inherently challenging, it clearly demonstrates the practical benefits of the flexibility in step-size selection provided by SPPM-inexact.
\begin{figure}[t]
    \centering
    \includegraphics[width=0.6\linewidth]{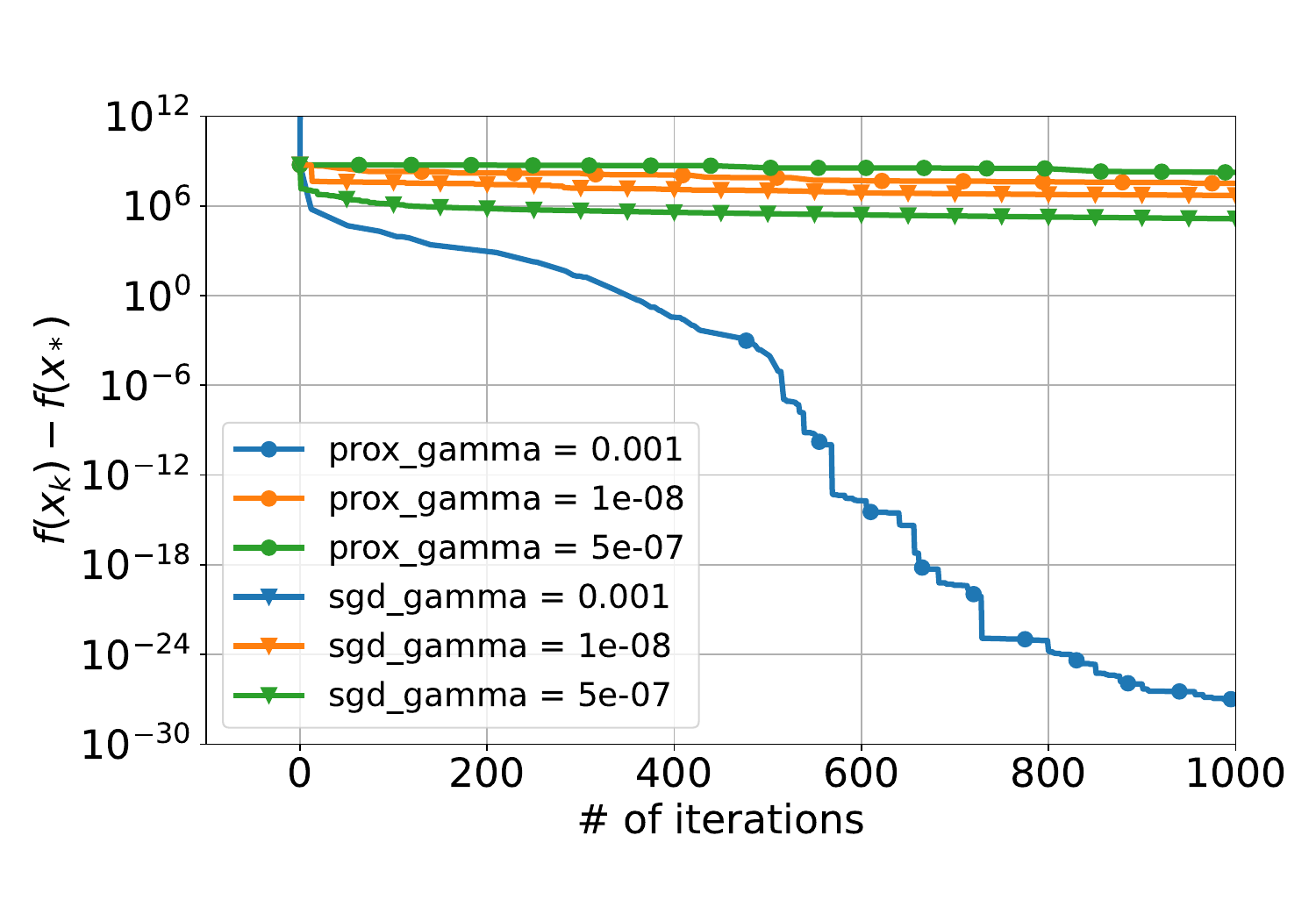}
    \caption{Comparison between SPPM-inexact and SGD with constant stepsizes \(\gamma\). }
    \label{fig:comparison_sppm_sgd}
\end{figure}

\subsection{Further Experiments}

In addition to the main comparisons, we further examine the robustness of SPPM-inexact with respect to step-size selection and initialization. These experiments complement the previous results by providing deeper insight into the practical behavior of the method.

\subsubsection{Impact of Different Stepsizes}

We next investigate the impact of different stepsizes on the convergence of the proposed method. Specifically, we aim to demonstrate that the method converges for any positive stepsize, provided an appropriately chosen tolerance for the inner solver is used, and that larger stepsizes lead to faster convergence rates. To validate this, we analyze three different values of \(s\), namely \(s = 2, 3, 4\). For each case, five different stepsizes are tested: \(\gamma = 0.1, 1, 10, 100, 1000\). The number of functions is set to \(n = 1000\), the dimension to \(d = 100\), and an inexact solver is employed with a stopping criterion based on the squared norm of the gradient, with an accuracy threshold of \(10^{-12}\): 
$
\|\nabla \Psi_k(\hat{x}_k)\|^2 \leq 10^{-12}.
$
As shown in Figure~\ref{fig:linereg1}, the method converges for all chosen stepsizes. While larger stepsizes accelerate convergence in terms of the number of iterations, they also increase the computational cost per iteration. These observations confirm our theoretical findings.

\subsubsection{Dependence on the Initial Point}

Finally, we study the dependence of convergence on the initial point. The objective is to verify that the convergence of SPPM-inexact is independent of the initial point and to assess the tightness of the theoretical analysis, i.e., whether the convergence rate depends on the distance between the initial and optimal points. For this experiment, we analyze three different values of \(s\) (\(s = 2, 3, 4\)) and, for each case, randomly select five different initial points with norms \(\|x_0\| = 0.1, 1, 10, 50, 100\), while keeping all other parameters unchanged. As shown in Figure~\ref{fig:linereg2}, the convergence rates are nearly identical across all cases, suggesting that the upper bound 
$
\|x_k - x_{k+1}\|^2 \leq \|x_0 - x_*\|^2
$ 
provided in Lemma \ref{eq:iterates_monotonicity} may be overly restrictive for certain problems. Further investigation of this effect is a promising direction for future research.

It is worth noting that increasing the parameter \(s\), which influences the problem formulation, makes the problem more challenging to solve. This is because the parameters \(L_0\) and \(L_1\) increase as \(s\) grows. This observation aligns with our theoretical understanding.

\begin{figure*}[t]
	\centering
	\includegraphics[width=0.326\textwidth]{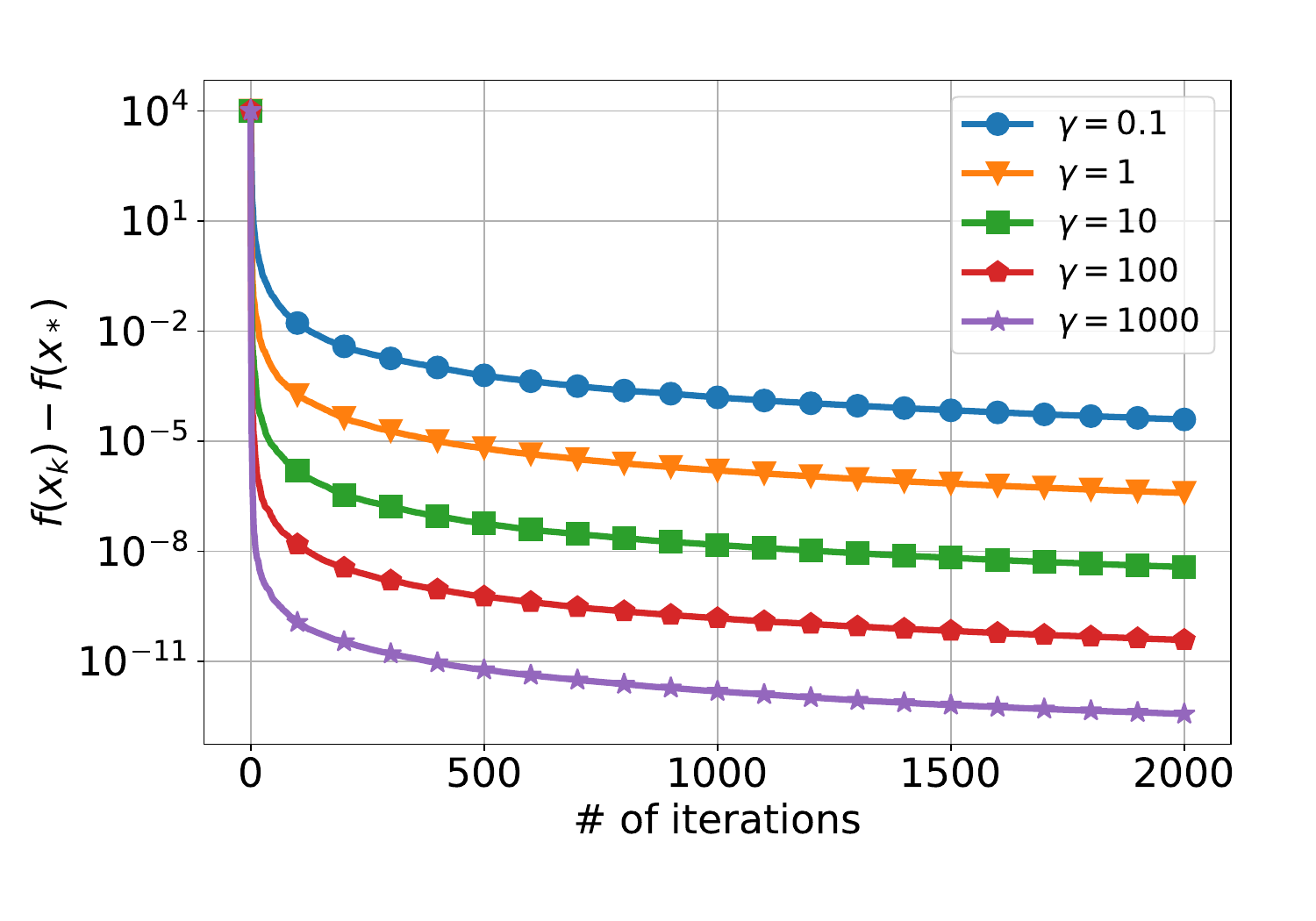}
	\includegraphics[width=0.326\textwidth]{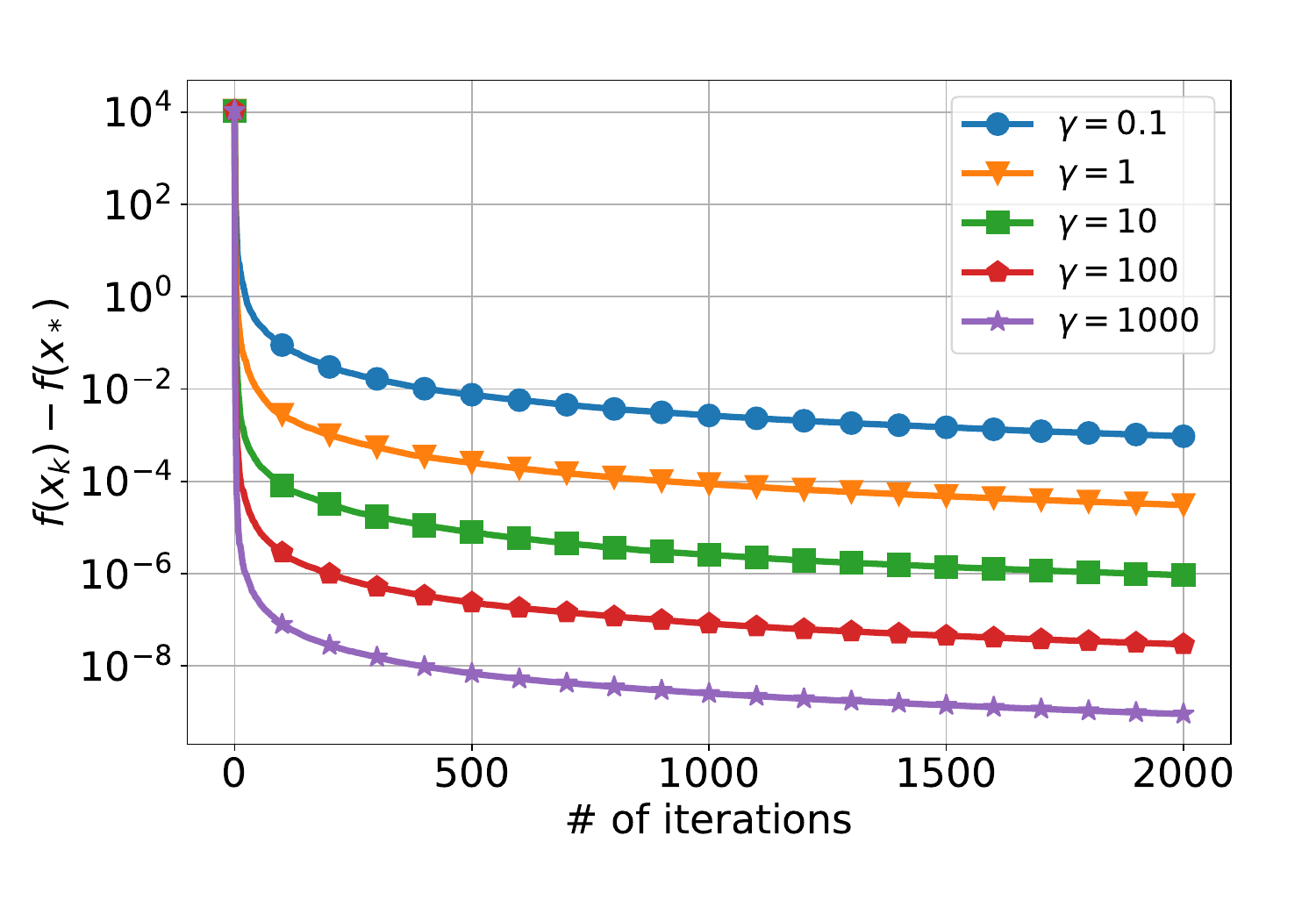}
	\includegraphics[width=0.326\textwidth]{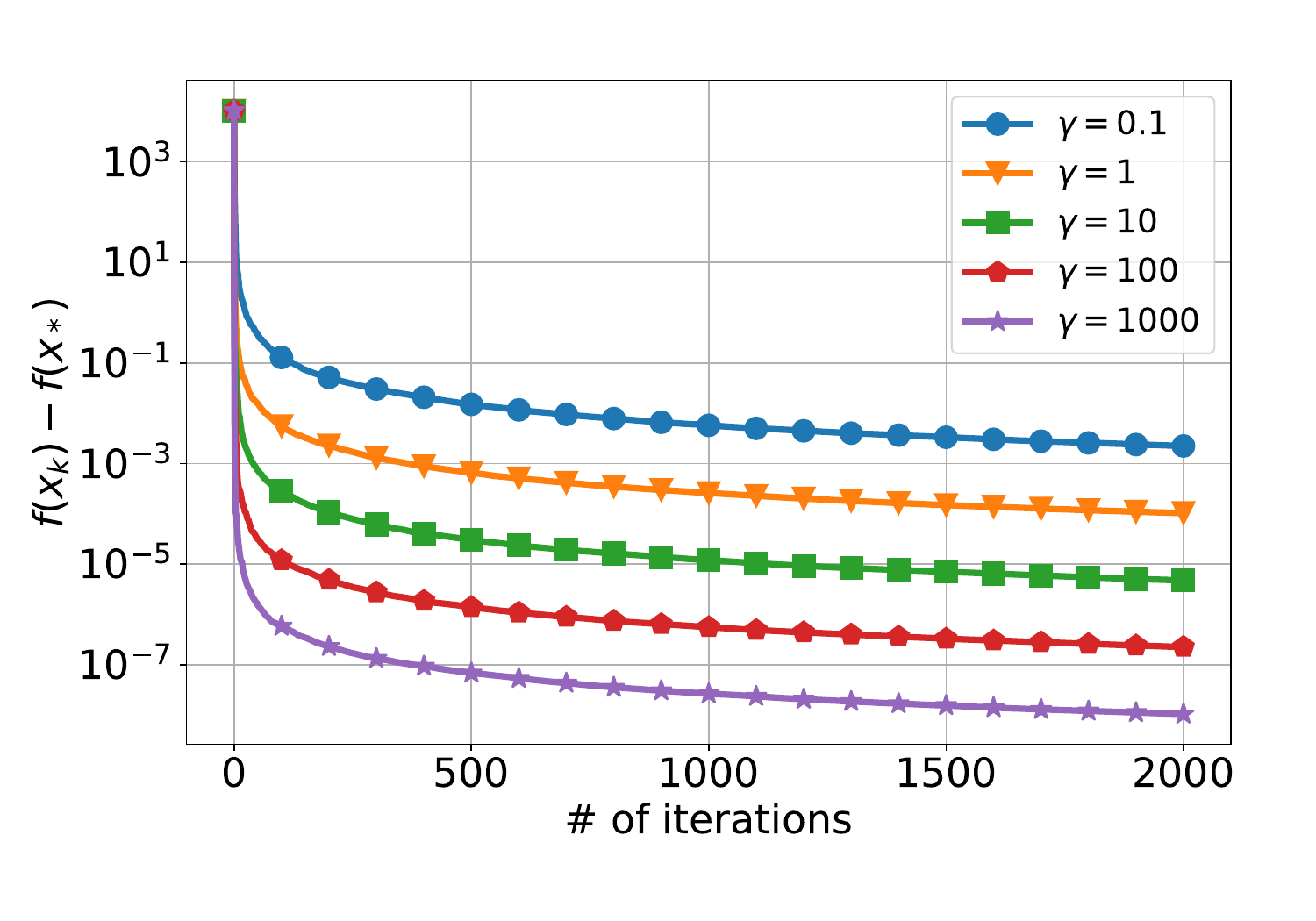}
	\vspace{-0.25cm}
	\caption{Convergence behavior of SPPM-inexact with different stepsizes.}\label{fig:linereg1}
\end{figure*}

\begin{figure*}[t]
	\centering
	\includegraphics[width=0.326\textwidth]{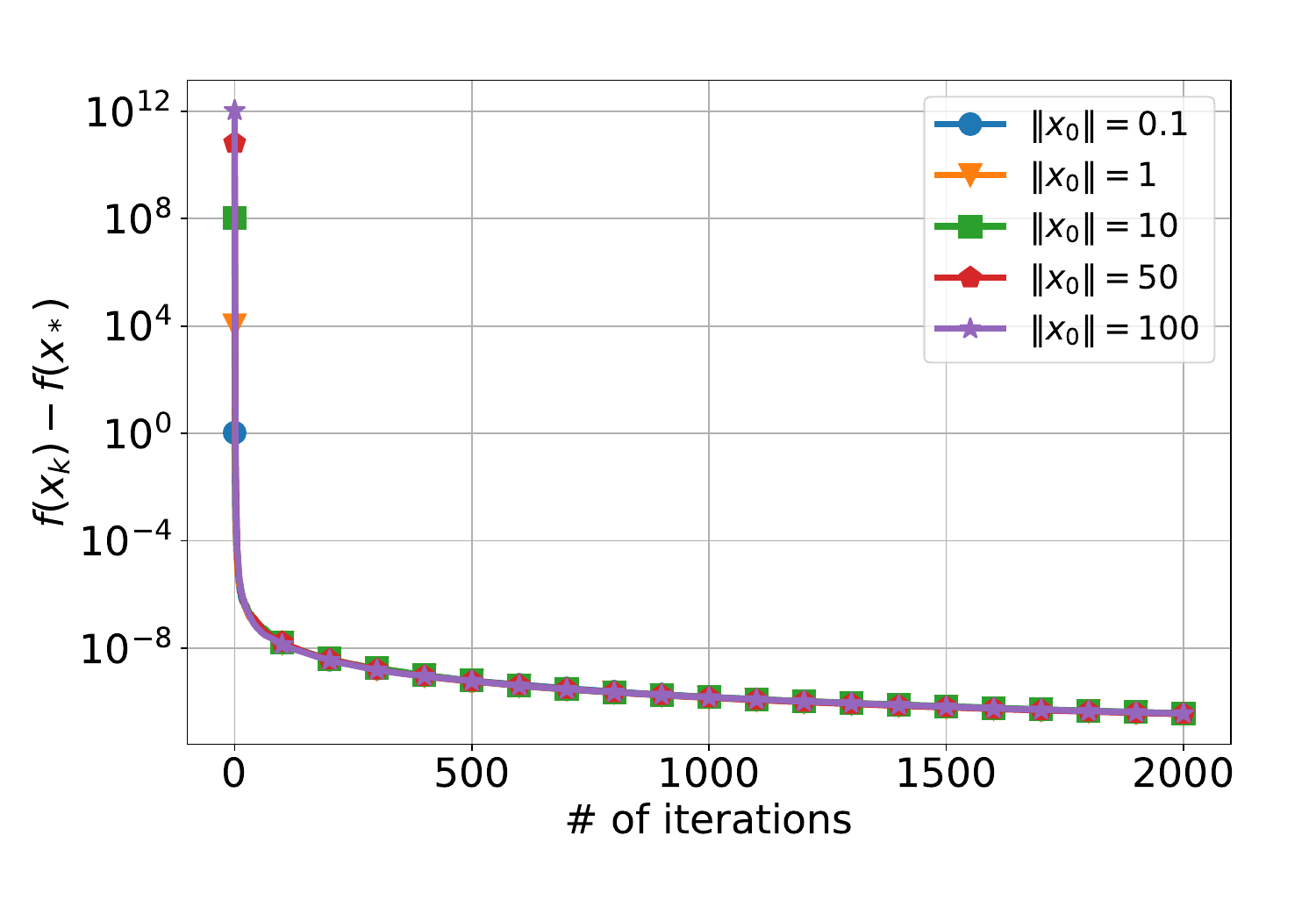}
	\includegraphics[width=0.326\textwidth]{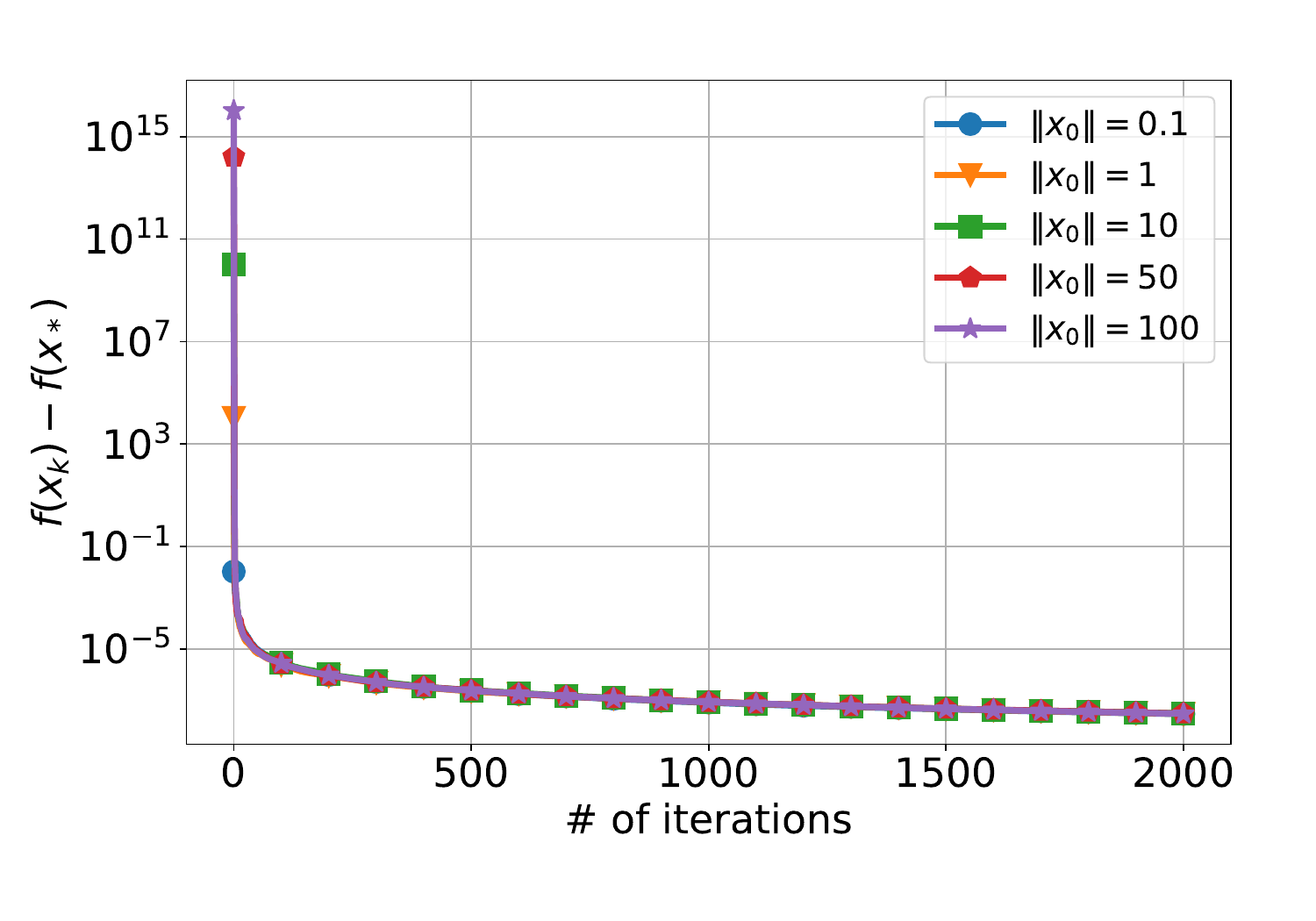}
	\includegraphics[width=0.326\textwidth]{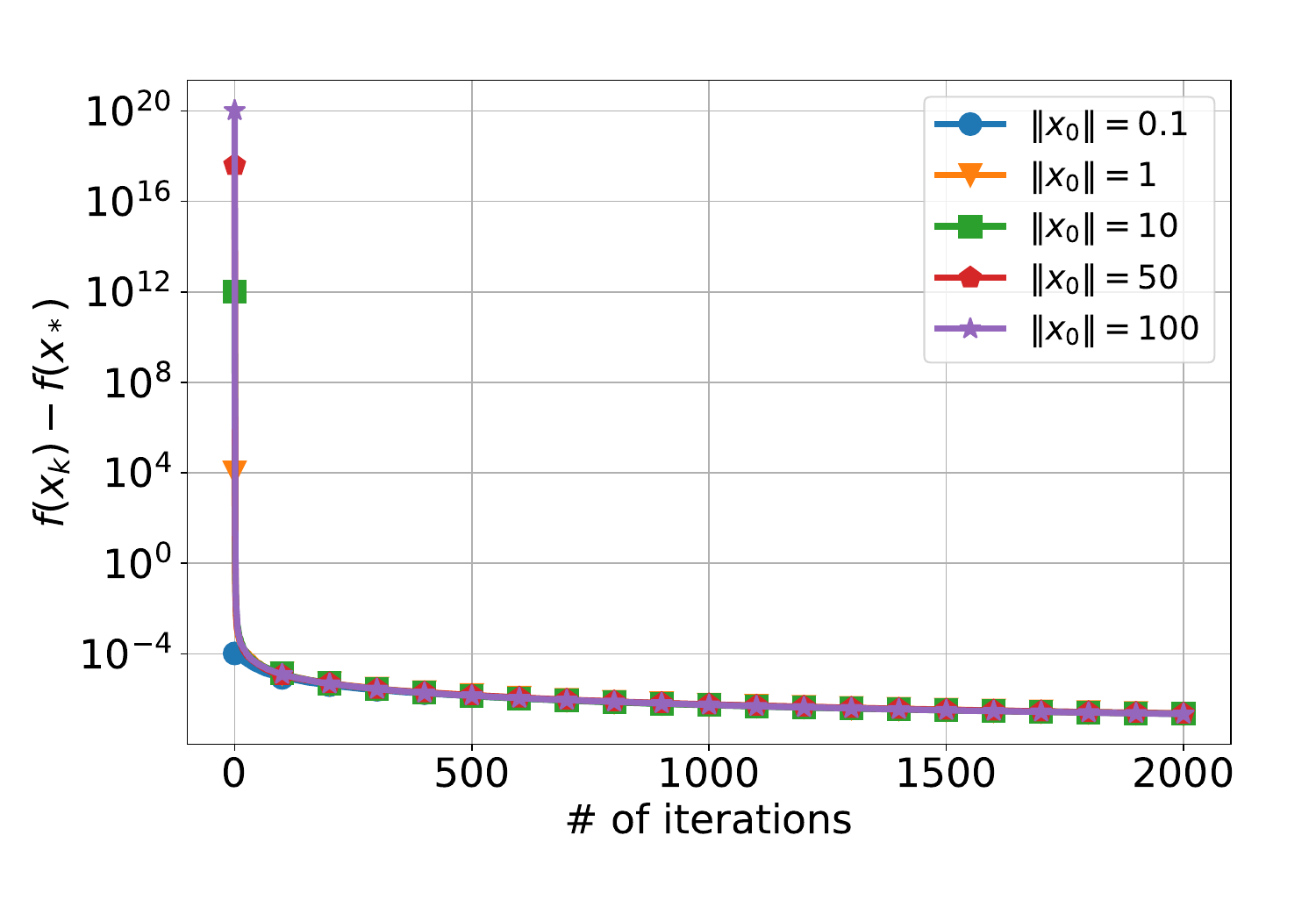}
	\vspace{-0.25cm}
	\caption{Convergence behavior of SPPM-inexact with different starting points.}\label{fig:linereg2}
\end{figure*}

\section{Conclusion}

In this work, we analyze SPPM methods under assumptions beyond standard Lipschitz smoothness. We introduce a generalized \(\phi\)-assumption that covers many special cases and provide convergence guarantees for both strongly convex and general settings in the interpolation regime. Additionally, we study convergence under the expected similarity assumption in strongly convex cases for both interpolation and non-interpolation regimes. Our analysis of stochastic methods beyond smoothness is a step forward in the understanding of their practical performance in ML. 
Further exploring SPPM in general convex settings
with expected similarity is an important direction for future research.

\section*{Acknowledgements}
This work was supported by funding from King Abdullah University of Science and Technology (KAUST):\\
i) KAUST Baseline Research Scheme,\\
ii) Center of Excellence for Generative AI (award no.\ 5940),\\
iii) Competitive Research Grant (CRG) Program (award no.\ 6460),\\
iv) SDAIA-KAUST Center of Excellence in Data Science and Artificial Intelligence (SDAIA-KAUST AI).

\bibliographystyle{icml2025}
\bibliography{paperbib}

\newpage
\appendix

{\noindent\huge \textbf{Appendix}}

\section{Fundamental Lemmas}

We define the Bregman divergence of a function $g$ as 
\begin{equation}
D_g(x, y):= g(x)-g(y)-\langle \nabla g(y),x-y\rangle\quad\forall x,y\in\mathbb{R}^d.\label{eqbreg}
\end{equation}

\begin{lemma}
	Let Assumptions \ref{assumption:differentiability}, \ref{assumption:convexity} and \ref{assumption:interpolation} hold. Then for any stepsize $\gamma>0$, the iterates of SPPM satisfy, for every $k\geq 0$,
	\begin{equation}
		\langle x_{k+1} - x_*, x_k - x_{k+1}\rangle \geq \gamma(f_{\xi_k}(x_{k+1}) - f_{\xi_k}(x_*)) \geq 0.\label{eq:convexity_interpolation_ineq}
	\end{equation}
\end{lemma}
\begin{proof}
	Since  
	\[
	x_{k+1} = \arg\min_{y \in \mathbb{R}^d} \left( f_{\xi_k}(y) + \frac{1}{2\gamma}\|y - x_k\|^2 \right),
	\]
	the first-order optimality condition implies that: 
	\begin{equation}
		\nabla f_{\xi_k}(x_{k+1}) + \frac{1}{\gamma}(x_{k+1}-x_k) = 0. \label{eq:opt_cond}
	\end{equation}
	Since \( f_{\xi_k} \) is differentiable and convex, we obtain the following inequality:  
	\begin{align*}
		f_{\xi_k}(x_*) &\geq f_{\xi_k}(x_{k+1}) + \langle\nabla f_{\xi_k}(x_{k+1}),x_*-x_{k+1}\rangle \\
					   &\overset{\eqref{eq:opt_cond}}{=} f_{\xi_k}(x_{k+1}) + \frac{1}{\gamma}\langle x_k-x_{k+1},x_*-x_{k+1} \rangle.
	\end{align*}
	Rearranging this expression yields the first inequality in \eqref{eq:convexity_interpolation_ineq}. Furthermore, we also have
	\begin{equation*}
		f_{\xi_k}(x_{k+1}) \geq \inf_{x\in\mathbb{R}^d} f_{\xi_k}(x) = f_{\xi_k}(x_*),
	\end{equation*}
	which follows from Assumption \ref{assumption:interpolation}.  
\end{proof}

\subsection{Proof of Lemma \ref{def:phi_bregman}}
	\begin{align*}
		f_{\xi}(x) - f_{\xi}(y) &= \int_0^1 \langle \nabla f_{\xi}(y + t(x - y)), x - y \rangle \, dt \\
		&= \int_0^1 \langle \nabla f_{\xi}(y + t(x - y)) - \nabla f_{\xi}(y), x - y \rangle \, dt + \langle \nabla f_{\xi}(y), x - y \rangle.
	\end{align*}
	Moving $\langle \nabla f_{\xi}(y), x - y \rangle$ to the left-hand side and using Cauchy--Schwarz inequality, we get
	\begin{align*}
		D_{f_{\xi}}(x, y) &= \int_0^1 \langle \nabla f_{\xi}(y + t(x - y)) - \nabla f_{\xi}(y), x - y \rangle \, dt \\
		&\leq \int_0^1 \| \nabla f_{\xi}(y + t(x - y)) - \nabla f_{\xi}(y) \| \, \| x - y \| \, dt.
	\end{align*}	
	Using $\phi$-smoothness of $f_{\xi}$ and that $\phi$ is nondecreasing on both variables, we get
	\begin{align} 
		\nonumber
		D_{f_{\xi}}(x, y) &\leq \int_0^1 ( \phi(t\|x-y\|,\|\nabla f_{\xi}(y)\|) t \| x - y \|^2 \, dt \\
		\nonumber
		&\leq  \int_0^1 ( \phi(\|x-y\|,\|\nabla f_{\xi}(y)\|) t \| x - y \|^2 \, dt \\
		&= \frac{\phi(\|x-y\|,\|\nabla f_{\xi}(y)\|)}{2}\|x-y\|^2.
	\end{align}

\subsection{Proof of Lemma~\ref{lemma:connection}}
From Proposition~3.2 (Point~2) in \citet{chen2023generalized}, the function \( f_{\xi} \colon \mathbb{R}^d \to \mathbb{R} \) is symmetrically \((L_0, L_1)\)-smooth if and only if for all \( x, y \in \mathbb{R}^d \),
\begin{align*}
    \|\nabla f_{\xi}(x) - \nabla f_{\xi}(y)\| \leq \left( L_0 + L_1 \|\nabla f_{\xi}(y)\| \right) \exp\left(L_1 \|x - y\|\right) \|x - y\|.
\end{align*}
Hence, defining
\begin{equation*}
\phi\left(\|x - y\|, \|\nabla f_{\xi}(y)\|\right) := \left( L_0 + L_1 \|\nabla f_{\xi}(y)\| \right) \exp\left(L_1 \|x - y\|\right),
\end{equation*}
we obtain
\begin{equation*}
\|\nabla f_{\xi}(x) - \nabla f_{\xi}(y)\| \leq \phi\left(\|x - y\|, \|\nabla f_{\xi}(y)\|\right) \|x - y\|,
\end{equation*}
which confirms that \( f_{\xi} \) is \( \phi \)-smooth according to Assumption~\ref{phi-smoothness}.

\subsection{Proof of Lemma~\ref{lemma:connection1}}
From Proposition~3.2 (Point~1) in \citet{chen2023generalized}, the function \( f_{\xi} \colon \mathbb{R}^d \to \mathbb{R} \) is said to satisfy the \(\alpha\)-symmetric generalized-smoothness condition if and only if for all \( x, y \in \mathbb{R}^d \),
\begin{align*}
    \|\nabla f_{\xi}(x) - \nabla f_{\xi}(y)\| 
    \leq \|x - y\| \left( K_0 + K_1 \|\nabla f_{\xi}(y)\|^\alpha + K_2 \|x - y\|^{\frac{\alpha}{1 - \alpha}} \right),
\end{align*}
where the constants \( K_0, K_1, K_2 \) are defined as:
\begin{align*}
    K_0 &:= L_0\left(2^{\frac{\alpha^2}{1 - \alpha}} + 1\right), \\
    K_1 &:= L_1 \cdot 2^{\frac{\alpha^2}{1 - \alpha}} \cdot 3^\alpha, \\
    K_2 &:= L_1^{\frac{1}{1 - \alpha}} \cdot 2^{\frac{\alpha^2}{1 - \alpha}} \cdot 3^\alpha (1 - \alpha)^{\frac{\alpha}{1 - \alpha}}.
\end{align*}
Hence, defining
\begin{equation*}
\phi\left(\|x - y\|, \|\nabla f_{\xi}(y)\|\right) := K_0 + K_1 \|\nabla f_{\xi}(y)\|^\alpha + K_2 \|x - y\|^{\frac{\alpha}{1 - \alpha}},
\end{equation*}
we obtain
\begin{equation*}
\|\nabla f_{\xi}(x) - \nabla f_{\xi}(y)\| 
\leq \phi\left(\|x - y\|, \|\nabla f_{\xi}(y)\|\right) \|x - y\|,
\end{equation*}
which confirms that \( f_{\xi} \) is \( \phi \)-smooth in the sense of Assumption~\ref{phi-smoothness}.

\subsection{Proof of Lemma \ref{eq:iterates_monotonicity}}

Let $k\geq 0$. 
	Since \( x_{k+1} = \arg\min_{y \in \mathbb{R}^d} \left( f_{\xi_k}(y) + \frac{1}{2\gamma}\|y - x_k\|^2 \right) \), we have
	\begin{equation*}
		f_{\xi_k}(x_{k+1}) + \frac{1}{2\gamma}\|x_k - x_{k+1}\|^2 \leq f_{\xi_k}(x_*) + \frac{1}{2\gamma}\|x_k - x_*\|^2,
	\end{equation*}	
	Rearranging the terms, we obtain
	\begin{equation*}
		\|x_k - x_{k+1}\|^2 \leq \|x_k - x_*\|^2 - 2\gamma \left( f_{\xi_k}(x_{k+1}) - f_{\xi_k}(x_*) \right).
	\end{equation*}
	Since \( f_{\xi_k}(x_{k+1}) - f_{\xi_k}(x_*) \geq 0 \), it follows that
	\begin{equation*}
		\|x_k - x_{k+1}\|^2 \leq \|x_k - x_*\|^2.
	\end{equation*}
	On the other hand, due to the nonexpansiveness of the proximal operator, we know that
	\begin{equation*}
		\|x_k - x_*\|^2 \leq \|x_{k-1} - x_*\|^2 \quad \text{for any } k \geq 1.
	\end{equation*}
	By applying this recursively, we obtain
	\begin{equation*}
		\|x_k - x_*\|^2 \leq \|x_0 - x_*\|^2.
	\end{equation*}
	Combining the above inequalities, we arrive at the desired result.

\subsection{Proof of Lemma \ref{lemma42}}

Let $k\geq 0$.
	Applying \eqref{def:phi_bregman} to the iterates, we obtain
	\begin{equation*}
		f_{\xi_k}(x_k) \leq f_{\xi_k}(x_{k+1}) + \langle \nabla f_{\xi_k}(x_{k+1}), x_k - x_{k+1} \rangle
		+ \frac{\phi\big(\|x_k - x_{k+1}\|, \|\nabla f(x_{k+1})\|\big)}{2} \|x_k - x_{k+1}\|^2.    
	\end{equation*}
	Using \eqref{eqoptc}
	we obtain
	\begin{equation*}
		f_{\xi_k}(x_k) \leq f_{\xi_k}(x_{k+1}) + \frac{1}{\gamma} \|x_k - x_{k+1}\|^2 + \frac{\phi\left(\|x_k - x_{k+1}\|, \frac{1}{\gamma} \|x_k - x_{k+1}\|\right) }{2}\|x_k - x_{k+1}\|^2.
	\end{equation*}
	Finally, applying \eqref{eq:iterates_monotonicity} and rearranging terms, we obtain
	\begin{equation*}
		f_{\xi_k}(x_k) - f_{\xi_k}(x_{k+1}) \leq  \left(\frac{1}{\gamma} +  \frac{\phi\left(\|x_0 - x_{*}\|, \frac{1}{\gamma} \|x_0 - x_{*}\|\right)}{2}\right) \|x_k - x_{k+1}\|^2.
	\end{equation*}

\section{Proof of Theorems \ref{theo43}  and \ref{theo44}} 

\subsection{Proof of Theorem \ref{theo43}}

For convenience, we restate Theorem \ref{theo43} here:
\begin{theorem}
	Let Assumptions \ref{assumption:differentiability} (Differentiability), \ref{assumption:convexity} (Convexity), \ref{assumption:interpolation} (Interpolation) and \ref{phi-smoothness} ($\phi$-smoothness) hold.   Then for any stepsize $\gamma>0$ we have, for every $k\geq 0$,
	\begin{equation}
		\mathbb{E}[f(\hat{x}_k)] - f(x_*) \leq  \frac{\phi(\|x_0-x_*\|,\frac{1}{\gamma}\|x_0-x_*\|) + \frac{2}{\gamma}}{2k}\|x_0 - x_*\|^2,
	\end{equation}
	where $\hat{x}_k$ is a vector chosen from the set of iterates {$x_0$, . . . , $x_{k-1}$} uniformly at random.
	
	If additionally Assumption \ref{assumption:strong_convexity} holds, then for any stepsize $\gamma>0$ we have, for every $k\geq 0$,
	\begin{equation}
		\mathbb{E} \big[ \| x_{k} - x_* \|^2 \big] \leq \left( 1 - \frac{\mu}{\frac{2}{\gamma} + \phi(\|x_0-x_*\|,\frac{1}{\gamma}\|x_0-x_*\|)}\right)^k \| x_0 - x_* \|^2.
	\end{equation}
\end{theorem}
\begin{proof}	
	We have
	\begin{align}
		\nonumber
		\| x_{k+1} - x_* \|^2 &= \| x_{k} - x_{*} - (x_k - x_{k+1}) \|^2 \\
		\nonumber
		&= \| x_{k} - x_* \|^2 - 2 \langle x_{k} - x_*, x_k - x_{k+1} \rangle + \| x_k - x_{k+1} \|^2 \\
		\nonumber
		&= \| x_{k} - x_* \|^2 - 2 \langle x_{k} - x_{k+1} + x_{k+1} - x_*, x_k - x_{k+1} \rangle + \| x_k - x_{k+1} \|^2 \\
		&= \| x_{k} - x_* \|^2 - 2 \langle x_{k+1} - x_*, x_k - x_{k+1} \rangle - \| x_k - x_{k+1} \|^2.\label{eq:standard_step}
	\end{align}
	Using \eqref{eq:convexity_interpolation_ineq} and \eqref{eq:iterates_lower_bound}, we obtain
	\begin{equation*}
		\| x_{k+1} - x_* \|^2 \leq \|x_k - x_*\|^2 - 2\gamma(f_{\xi_k}(x_{k+1}) - f_{\xi_k}(x_*)) - \frac{1}{\frac{1}{\gamma} + \frac{\phi\left(\|x_0 - x_{*}\|, \frac{1}{\gamma} \|x_0 - x_{*}\|\right)}{2}}(f_{\xi_k}(x_k) - f_{\xi_k}(x_{k+1})).
	\end{equation*}
	Since \( 2\gamma > \frac{1}{\frac{1}{\gamma} + \frac{\phi\left(\|x_0 - x_{*}\|, \frac{1}{\gamma} \|x_0 - x_{*}\|\right)}{2}} \), we obtain
	\begin{equation*}
		\| x_{k+1} - x_* \|^2 \leq \|x_k - x_*\|^2 - \frac{1}{\frac{1}{\gamma} + \frac{\phi\left(\|x_0 - x_{*}\|, \frac{1}{\gamma} \|x_0 - x_{*}\|\right)}{2}}(f_{\xi_k}(x_k) - f_{\xi_k}(x_*)).
	\end{equation*}
	Let \( \mathcal{F}_k \) denote the \( \sigma \)-algebra generated by the collection of random variables \( (x_0, \ldots, x_k) \). Taking the expectation conditioned on \( \mathcal{F}_k \), we have
	\begin{equation}
		\mathbb{E}[\| x_{k+1} - x_* \|^2 \mid \mathcal{F}_k] \leq \|x_k - x_*\|^2 - \frac{1}{\frac{1}{\gamma} + \frac{\phi\left(\|x_0 - x_{*}\|, \frac{1}{\gamma} \|x_0 - x_{*}\|\right)}{2}}(f(x_k) - f(x_*)).\label{eq:phi_main_inequality}
	\end{equation}
	Taking full expectation, we get
	\begin{equation*}
		\mathbb{E}[\| x_{k+1} - x_* \|^2] \leq \mathbb{E}[\|x_k - x_*\|^2] - \frac{1}{\frac{1}{\gamma} + \frac{\phi\left(\|x_0 - x_{*}\|, \frac{1}{\gamma} \|x_0 - x_{*}\|\right)}{2}}(\mathbb{E}[f(x_k)] - f(x_*)).
	\end{equation*}
	By summing up the inequalities telescopically for \( t = 0, \ldots, k \), we obtain
	\begin{align*}
		\sum_{t=0}^k \mathbb{E}[f(x_t)] - f(x_*) &\leq \left(\frac{1}{\gamma} + \frac{\phi\left(\|x_0 - x_{*}\|, \frac{1}{\gamma} \|x_0 - x_{*}\|\right)}{2}\right)(\|x_0 - x_*\|^2 - \mathbb{E}[\|x_{k+1} - x_*\|^2]) \\
		&\leq \left(\frac{1}{\gamma} + \frac{\phi\left(\|x_0 - x_{*}\|, \frac{1}{\gamma} \|x_0 - x_{*}\|\right)}{2}\right)\|x_0 - x_*\|^2.
	\end{align*}
	Notice that:
	\begin{align*}
		\mathbb{E}[f(\hat{x}_{k+1})] &= \mathbb{E}\left[\mathbb{E}[f(\hat{x}_{k+1}) \mid \mathcal{F}_k]\right] = \mathbb{E}\left[\frac{1}{k+1} \sum_{t=0}^k f(x_t)\right] = \frac{1}{k+1} \sum_{t=0}^k \mathbb{E}[f(x_t)].
	\end{align*}
	Thus, we have
	\begin{equation*}
		\mathbb{E}[f(\hat{x}_{k+1})] - f(x_*) \leq \frac{\frac{2}{\gamma} + \phi\left(\|x_0 - x_{*}\|, \frac{1}{\gamma} \|x_0 - x_{*}\|\right)}{2(k+1)}\|x_0 - x_*\|^2.
	\end{equation*}
	If we assume \ref{assumption:strong_convexity}, then in step \eqref{eq:phi_main_inequality}, applying the strong convexity of \( f \), we get
	\begin{align*}
		\mathbb{E}[\| x_{k+1} - x_* \|^2 \mid \mathcal{F}_k] &\leq \|x_k - x_*\|^2 - \frac{\mu}{\frac{2}{\gamma} + \phi\left(\|x_0 - x_{*}\|, \frac{1}{\gamma} \|x_0 - x_{*}\|\right)}\|x_k - x_*\|^2 \\
		&\leq \left(1 - \frac{\mu}{\frac{2}{\gamma} + \phi\left(\|x_0 - x_{*}\|, \frac{1}{\gamma} \|x_0 - x_{*}\|\right)}\right)\|x_k - x_*\|^2.
	\end{align*}
	Taking full expectation, we obtain
	\begin{equation*}
		\mathbb{E}[\| x_{k+1} - x_* \|^2] \leq \left(1 - \frac{\mu}{\frac{2}{\gamma} + \phi\left(\|x_0 - x_{*}\|, \frac{1}{\gamma} \|x_0 - x_{*}\|\right)}\right)\mathbb{E}[\|x_k - x_*\|^2].
	\end{equation*}
	Applying this recursively, we get
	\begin{equation*}
		\mathbb{E}[\| x_{k+1} - x_* \|^2] \leq \left(1 - \frac{\mu}{\frac{2}{\gamma} + \phi\left(\|x_0 - x_{*}\|, \frac{1}{\gamma} \|x_0 - x_{*}\|\right)}\right)^{k+1}\|x_0 - x_*\|^2.
	\end{equation*}
\end{proof}

\subsection{Proof of Theorem \ref{theo44}}

For convenience, we restate Theorem \ref{theo44} here:
\begin{theorem}
	Let Assumptions \ref{assumption:differentiability} (Differentiability), \ref{assumption:convexity} (Convexity), \ref{assumption:interpolation} (Interpolation) and \ref{phi-smoothness} ($\phi$-smoothness) hold. Consider 
SPPM-inexact with every inexact prox satisfying Assumption \ref{assumption:inexactness_condition}. If \( T \) is chosen sufficiently large such that
	$\frac{\eta \gamma^2}{T^\alpha}\leq c < 1$, then, for any stepsize \( \gamma > 0 \) the iterates of 
	SPPM-inexact
	satisfy, for every $k\geq 0$,
	\begin{equation*}
		\mathbb{E}[f(\hat{x}_k)] - f(x_*) \leq \frac{\phi(\|x_0-x_*\|,\frac{1}{\gamma}\|x_0-x_*\|) + \frac{2}{\gamma}}{2k\left(1-c\right)}\|x_0 - x_*\|^2.
	\end{equation*}
	If in addition Assumption \ref{assumption:strong_convexity} holds, then for any \( \gamma > 0 \), the iterates of 
SPPM-inexact satisfy, for every $k\geq 0$,
	\begin{equation*}
		\mathbb{E}\big[\|x_k - x_*\|^2\big] \leq \left( 1 - \frac{(1-c)\mu}{\frac{2}{\gamma} + \phi(\|x_0-x_*\|,\frac{1}{\gamma}\|x_0-x_*\|)}\right)^{k}\|x_{0} - x_*\|^2.
	\end{equation*}
\end{theorem}

\begin{proof}
	We have
	\begin{equation*}
		\|x_{k+1} -  x_*\| = \| x_{k} - x_* \|^2 - 2 \langle x_{k+1} - x_*, x_k - x_{k+1} \rangle - \|x_k - x_{k+1} \|^2.
	\end{equation*}
	Substituting \( x_{k+1} = x_k - \gamma \nabla f_{\xi_k}(\hat{x}_k) \), we obtain
	\begin{align*}
		\nonumber
		\|x_{k+1} - x_*\|^2 
		&= \|x_k - x_*\|^2 - \|x_{k+1} - x_k\|^2 - 2\gamma \langle \nabla f_{\xi_k}(\hat{x}_k), x_{k+1} - x_* \rangle \\
		\nonumber
		&= \|x_k - x_*\|^2 - \|x_{k+1} - x_k\|^2 - 2\gamma \langle \nabla f_{\xi_k}(\hat{x}_k), x_{k+1} - \hat{x}_k + \hat{x}_k - x_* \rangle \\
		&= \|x_k - x_*\|^2 - \|x_{k+1} - x_k\|^2 - 2\gamma \langle \nabla f_{\xi_k}(\hat{x}_k), \hat{x}_k - x_* \rangle  - 2\gamma \langle \nabla f_{\xi_k}(\hat{x}_k), x_{k+1} - \hat{x}_k \rangle.
	\end{align*}
	Using the identity \( -2\langle a, b \rangle = -\|a + b\|^2 + \|a\|^2 + \|b\|^2 \), we rewrite the expression as
	\begin{align*}
		\nonumber
		\|x_{k+1} - x_*\|^2 
		&= \|x_k - x_*\|^2 -\|x_{k+1}-x_k\|^2 - 2\gamma \langle \nabla f_{\xi_k}(\hat{x}_k), \hat{x}_k - x_* \rangle  \\
		&\quad - \|x_{k+1} + \gamma \nabla f_{\xi_k}(\hat{x}_k) - \hat{x}_k\|^2 + \|\gamma \nabla  f_{\xi_k}(\hat{x}_k)\|^2 + \|x_{k+1} - \hat{x}_k\|^2.
	\end{align*}
	Noting that \( \|x_{k+1} + \gamma \nabla f_{\xi_k}(\hat{x}_k) - x_k\|^2 = \|x_k - \hat{x}_k\|^2 \) and \( \|\gamma \nabla f_{\xi_k}(\hat{x}_k)\|^2 = \|x_{k+1} - x_k\|^2 \), we simplify as
	\begin{align}
		\|x_{k+1} - x_*\|^2 
		&= \|x_k - x_*\|^2 - 2\gamma \langle \nabla f_{\xi_k}(\hat{x}_k), \hat{x}_k - x_* \rangle \notag - \|x_k - \hat{x}_k\|^2 + \|x_{k+1} - \hat{x}_k\|^2.
	\end{align}
	Using $ \|x_{k+1} - \hat{x}_k\|^2 = \|x_k - \gamma \nabla f_{\xi_k}(\hat{x}_k) - \hat{x}_k\|^2 = \|-\gamma \nabla \Psi_k(\hat{x}_k)\|^2 $, we derive
	\begin{align}
		\|x_{k+1} - x_*\|^2 
		&= \|x_k - x_*\|^2 - 2\gamma \langle \nabla f_{\xi_k}(\hat{x}_k), \hat{x}_k - x_* \rangle - \|x_k - \hat{x}_k\|^2 + \gamma^2 \|\nabla \Psi_k(\hat{x}_k)\|^2. \label{eq:start_ineq}
	\end{align}
	By convexity of \( f_{\xi_k} \), we have
	\begin{equation*}
		f_{\xi_k}(x_*) \geq f_{\xi_k}(\hat{x}_k) + \langle \nabla f_{\xi_k}(\hat{x}_k),x_*-\hat{x}_k \rangle,
	\end{equation*}
	which implies
	\begin{equation}
		\langle \nabla f_{\xi_k}(\hat{x}_k),\hat{x}_k - x_* \rangle \geq f_{\xi_k}(\hat{x}_k) - f_{\xi_k}(x_*) .\label{eq:conv_hat}
	\end{equation}
	Using \eqref{eq:conv_hat} in \eqref{eq:start_ineq}, we get
	\begin{equation}
		\|x_{k+1} - x_*\|^2 
		\leq \|x_k - x_*\|^2 - 2\gamma(f_{\xi_k}(\hat{x}_k) - f_{\xi_k}(x_*))  - \|x_k - \hat{x}_k\|^2 + \gamma^2 \|\nabla \Psi_k(\hat{x}_k)\|^2. \label{eq:cornerstone}
	\end{equation}
	Since $ x_{k}^{\Psi}$ is the minimizer of $\Psi_k(x)$, we have
	\begin{equation*}
		f_{\xi_k}(x_{k}^{\Psi}) + \frac{1}{2\gamma}\|x_{k}^{\Psi}-x_k\|^2 \leq  f_{\xi_k}(\hat{x}_{k}) + \frac{1}{2\gamma}\|\hat{x}_{k}-x_k\|^2,
	\end{equation*}
	from which it follows that
	\begin{equation*}
		f_{\xi_k}(\hat{x}_{k}) \geq f_{\xi_k}(x_{k}^{\Psi}) + \frac{1}{2\gamma}\|x_{k}^{\Psi}-x_k\|^2 - \frac{1}{2\gamma}\|\hat{x}_{k}-x_k\|^2.
	\end{equation*}
	Replacing it in \eqref{eq:cornerstone}, we obtain
	\begin{align*}
		\nonumber
		\|x_{k+1} - x_*\|^2 
		&\leq \|x_k - x_*\|^2 - 2\gamma(f_{\xi_k}(x_{k}^{\Psi}) - f_{\xi_k}(x_*)) -\|x_{k}^{\Psi}-x_k\|^2 + \|\hat{x}_{k}-x_k\|^2 -\|x_k - \hat{x}_k\|^2 \\
		&\quad+ \gamma^2 \|\nabla \Psi_k(\hat{x}_k)\|^2 \\
		& =  \|x_k - x_*\|^2 - 2\gamma(f_{\xi_k}(x_{k}^{\Psi}) - f_{\xi_k}(x_*)) -\|x_{k}^{\Psi}-x_k\|^2 + \gamma^2 \|\nabla \Psi_k(\hat{x}_k)\|^2.
	\end{align*}
	Using \eqref{eq:inexact_condition}, we get
	\begin{align}
		\nonumber
		\|x_{k+1} - x_*\|^2 &\leq \|x_k - x_*\|^2 - 2\gamma(f_{\xi_k}(x_{k}^{\Psi}) - f_{\xi_k}(x_*)) -\|x_{k}^{\Psi}-x_k\|^2 + \frac{\eta\gamma^2\|x_k-x_k^{\Psi}\|^2}{T^\alpha} \\
		&= \|x_k - x_*\|^2 - 2\gamma(f_{\xi_k}(x_{k}^{\Psi}) - f_{\xi_k}(x_*)) -(1-\frac{\eta\gamma^2}{T^{\alpha}})\|x_k-x_k^{\Psi}\|^2 .\label{eq:inexactness_main_inequality}
	\end{align}
	From the proximal step, we have
	\begin{equation*}
		\nonumber
		f_{\xi_{k}}(x_{k}^{\Psi}) + \frac{1}{2\gamma}\|x_k - x_{k}^{\Psi}\|^2 \leq f_{\xi_k}(x_*) + \frac{1}{2\gamma}\|x_* - x_k\|^2 .
	\end{equation*}
	Rearranging the terms, we get
	\begin{align*}
		\|x_k - x_{k}^{\Psi}\|^2 &\leq \|x_* - x_k\|^2 - 2\gamma(f_{\xi_k}(x_{k}^{\Psi}) - f_{\xi_k}(x_*)) \\
			&\leq \|x_* - x_k\|^2,
	\end{align*}
	and from the \eqref{eq:inexactness_main_inequality}, we know that $\|x_{k+1}-x_*\|^2 \leq \|x_k-x_*\|^2$, so using recursion, we get $\|x_k - x_{k}^{\Psi}\|^2 \leq \|x_0-x_*\|^2$, which means we can use \eqref{eq:iterates_lower_bound} in \eqref{eq:inexactness_main_inequality}:
	\begin{equation}
		\|x_{k+1} - x_*\|^2 \leq \|x_k - x_*\|^2 - 2\gamma(f_{\xi_k}(x_{k}^{\Psi}) - f_{\xi_k}(x_*)) - \frac{1-\frac{\eta\gamma^2}{T^{\alpha}}}{
			\frac{1}{\gamma} + \frac{ \phi\left(\|x_0 - x_{*}\|, \frac{1}{\gamma} \|x_0 - x_{*}\|\right)}{2}}\left(f_{\xi_k}(x_k) - f_{\xi_k}(x_{k}^{\Psi})\right),
	\end{equation}
	and since $2\gamma > \frac{1-\frac{\eta\gamma^2}{T^{\alpha}}}{
		\frac{1}{\gamma} + \frac{ \phi\left(\|x_0 - x_{*}\|, \frac{1}{\gamma} \|x_0 - x_{*}\|\right)}{2}}$, we can write
	\begin{align*}
		\|x_{k+1} - x_*\|^2 &\leq \|x_k - x_*\|^2 - \frac{1-\frac{\eta\gamma^2}{T^{\alpha}}}{
			\frac{1}{\gamma} + \frac{ \phi\left(\|x_0 - x_{*}\|, \frac{1}{\gamma} \|x_0 - x_{*}\|\right)}{2}}(f_{\xi_k}(x_{k}^{\Psi}) - f_{\xi_k}(x_*)) \\
		&\quad - \frac{1-\frac{\eta\gamma^2}{T^{\alpha}}}{
			\frac{1}{\gamma} + \frac{ \phi\left(\|x_0 - x_{*}\|, \frac{1}{\gamma} \|x_0 - x_{*}\|\right)}{2}}\left(f_{\xi_k}(x_k) - f_{\xi_k}(x_{k}^{\Psi})\right) \\ 
		&= \|x_k - x_*\|^2 - \frac{1-\frac{\eta\gamma^2}{T^{\alpha}}}{\frac{1}{\gamma} + \frac{ \phi\left(\|x_0 - x_{*}\|, \frac{1}{\gamma} \|x_0 - x_{*}\|\right)}{2}}\left(f_{\xi_k}(x_{k}) - f_{\xi_k}(x_*)\right).
	\end{align*}
	Taking the expectation conditioned on $\mathcal{F}_k$, we have
	\begin{equation}
		\mathbb{E}[\|x_{k+1} - x_*\|^2\mid \mathcal{F}_k] \leq \|x_k - x_*\|^2 - \frac{1-\frac{\eta\gamma^2}{T^{\alpha}}}{\frac{1}{\gamma} + \frac{ \phi\left(\|x_0 - x_{*}\|, \frac{1}{\gamma} \|x_0 - x_{*}\|\right)}{2}}\left(f(x_{k}) - f(x_*)\right)\label{eq:inexact_separation_point}.
	\end{equation}
	Now, by taking the full expectation, we get
	\begin{equation*}
		\mathbb{E}[\|x_{k+1} - x_*\|^2] \leq \mathbb{E}[\|x_k - x_*\|^2] - \frac{1-\frac{\eta\gamma^2}{T^{\alpha}}}{\frac{1}{\gamma} + \frac{ \phi\left(\|x_0 - x_{*}\|, \frac{1}{\gamma} \|x_0 - x_{*}\|\right)}{2}}\left(\mathbb{E}[f(x_{k})] - f(x_*)\right).
	\end{equation*}
	By summing up the inequalities telescopically for \( t = 0, \ldots, k \), we obtain
	\begin{align*}
		\sum_{t=0}^k \mathbb{E}[f(x_t)] - f(x_*) &\leq \frac{\frac{1}{\gamma} + \frac{\phi\left(\|x_0 - x_{*}\|, \frac{1}{\gamma} \|x_0 - x_{*}\|\right)}{2}}{1 -\frac{\eta\gamma^2}{T^{\alpha}}}(\|x_0 - x_*\|^2 - \mathbb{E}[\|x_{k+1} - x_*\|^2]) \\
		&\leq \frac{\frac{1}{\gamma} + \frac{\phi\left(\|x_0 - x_{*}\|, \frac{1}{\gamma} \|x_0 - x_{*}\|\right)}{2}}{1 -\frac{\eta\gamma^2}{T^{\alpha}}}\|x_0 - x_*\|^2.
	\end{align*}
	Notice that
	\begin{align*}
		\mathbb{E}[f(\hat{x}_{k+1})] &= \mathbb{E}\left[\mathbb{E}[f(\hat{x}_{k+1}) \mid \mathcal{F}_k]\right] = \mathbb{E}\left[\frac{1}{k+1} \sum_{t=0}^k f(x_t)\right] = \frac{1}{k+1} \sum_{t=0}^k \mathbb{E}[f(x_t)].
	\end{align*}
	Thus, we have
	\begin{align*}
		\mathbb{E}[f(\hat{x}_{k+1})] - f(x_*) &\leq \frac{\frac{2}{\gamma} + \phi\left(\|x_0 - x_{*}\|, \frac{1}{\gamma} \|x_0 - x_{*}\|\right)}{2(k+1)\left(1 - \frac{\eta\gamma^2}{T^{\alpha}}\right)}\|x_0 - x_*\|^2\\
		&\leq \frac{\frac{2}{\gamma} + \phi\left(\|x_0 - x_{*}\|, \frac{1}{\gamma} \|x_0 - x_{*}\|\right)}{2(k+1)\left(1 -c\right)}\|x_0 - x_*\|^2.
	\end{align*}
	If we assume \ref{assumption:strong_convexity}, then in step \eqref{eq:inexact_separation_point}, applying the strong convexity of \( f \), we get
	\begin{align*}
		\mathbb{E}[\| x_{k+1} - x_* \|^2 \mid \mathcal{F}_k] &\leq \|x_k - x_*\|^2 - \frac{\mu\left(1 - \frac{\eta\gamma^2}{T^{\alpha}}\right)}{\frac{2}{\gamma} + \phi\left(\|x_0 - x_{*}\|, \frac{1}{\gamma} \|x_0 - x_{*}\|\right)}\|x_k - x_*\|^2 \\
		&\leq \left(1 - \frac{\mu\left(1 - \frac{\eta\gamma^2}{T^{\alpha}}\right)}{\frac{2}{\gamma} + \phi\left(\|x_0 - x_{*}\|, \frac{1}{\gamma} \|x_0 - x_{*}\|\right)}\right)\|x_k - x_*\|^2.
	\end{align*}
	Taking the full expectation, we obtain
	\begin{equation*}
		\mathbb{E}[\| x_{k+1} - x_* \|^2] \leq \left(1 - \frac{\mu\left(1 - \frac{\eta\gamma^2}{T^{\alpha}}\right)}{\frac{2}{\gamma} + \phi\left(\|x_0 - x_{*}\|, \frac{1}{\gamma} \|x_0 - x_{*}\|\right)}\right)\mathbb{E}[\|x_k - x_*\|^2].
	\end{equation*}
	Applying this recursively, we get
	\begin{align*}
		\mathbb{E}[\| x_{k+1} - x_* \|^2] &\leq \left(1 - \frac{\mu\left(1 - \frac{\eta\gamma^2}{T^{\alpha}}\right)}{\frac{2}{\gamma} + \phi\left(\|x_0 - x_{*}\|, \frac{1}{\gamma} \|x_0 - x_{*}\|\right)}\right)^{k+1}\|x_0 - x_*\|^2\\
		&\leq \left(1 - \frac{\mu\left(1 -c\right)}{\frac{2}{\gamma} + \phi\left(\|x_0 - x_{*}\|, \frac{1}{\gamma} \|x_0 - x_{*}\|\right)}\right)^{k+1}\|x_0 - x_*\|^2.
	\end{align*}
\end{proof}

\subsubsection{Practical details (parameter-free Armijo on $\Psi_k$).}\label{sec:practical_details}
Fix $k$. Recall $\Psi_k$ and $x_k^{\Psi}$ from \eqref{eqpsi}. By the optimality condition,
\begin{align}\label{eq:prox-opt}
x_k^{\Psi} + \gamma \nabla f_{\xi_k}(x_k^{\Psi}) = x_k
\;\;\Longrightarrow\;\;
\|\nabla f_{\xi_k}(x_k^{\Psi})\| = \frac{1}{\gamma}\|x_k - x_k^{\Psi}\|.
\end{align}
From the proof of Theorem \ref{theo44},
\begin{align}\label{eq:Rbound}
\|x_k - x_k^{\Psi}\| \le \|x_0 - x_*\| =: R,
\end{align}
hence by \eqref{eq:prox-opt},
\begin{align}\label{eq:grad-at-center}
\|\nabla f_{\xi_k}(x_k^{\Psi})\| \le \frac{R}{\gamma}.
\end{align}
Let $B_k := \{ z : \|z - x_k^{\Psi}\| \le R \}$. For $y \in B_k$, by $\phi$-smoothness and \eqref{eq:grad-at-center},
\begin{align}
\|\nabla f_{\xi_k}(y) - \nabla f_{\xi_k}(x_k^{\Psi})\|
&\le \phi(\|y - x_k^{\Psi}\|, \|\nabla f_{\xi_k}(x_k^{\Psi})\|)\|y - x_k^{\Psi}\| \nonumber\\
&\le \phi(R, \tfrac{R}{\gamma})R.
\end{align}
Thus
\begin{align}\label{eq:G}
\|\nabla f_{\xi_k}(y)\| \le G := \frac{R}{\gamma} + \phi(R,\tfrac{R}{\gamma})R, 
\qquad y \in B_k.
\end{align}
Now for any $x,y \in B_k$,
\begin{align}
\|\nabla f_{\xi_k}(x) - \nabla f_{\xi_k}(y)\|
&\le \phi(\|x - y\|, \|\nabla f_{\xi_k}(y)\|)\|x - y\| \nonumber\\
&\le \phi(2R, G)\|x - y\|.
\end{align}
Since $\nabla \Psi_k(x) = \nabla f_{\xi_k}(x) + \tfrac{1}{\gamma}(x - x_k)$, we obtain
\begin{align}\label{eq:Lloc}
\|\nabla \Psi_k(x) - \nabla \Psi_k(y)\|
\le (\phi(2R,G) + \tfrac{1}{\gamma})\|x - y\| 
=: L_{\rm loc}\|x - y\|, \quad \forall x,y \in B_k.
\end{align}
Here $L_{\rm loc}$ is independent of $k$, and $\Psi_k$ is $\mu$–strongly convex with $\mu=\tfrac{1}{\gamma}$.

\medskip
\noindent\textbf{Implication for Assumption~\ref{assumption:inexactness_condition}.}
Since $\Psi_k$ is $\mu$-strongly convex and has a locally Lipschitz gradient on $B_k$, 
gradient descent on $\Psi_k$ starting at $x_k$ with \emph{Armijo backtracking} enjoys linear convergence. 
Hence there exist $\rho\in(0,1)$ and $C>0$ (independent of $k$) such that
\[
\|x^{t}-x_k^{\Psi}\|\le \rho^{t}\|x_k-x_k^{\Psi}\|,
\qquad
\|\nabla\Psi_k(x^{t})\|\le C\,\rho^{t}\|x_k-x_k^{\Psi}\|.
\]
Therefore, for any $\alpha>0$,
\[
\|\nabla\Psi_k(x^{T})\|^2 \le \eta\,\frac{\|x_k-x_k^{\Psi}\|^2}{T^{\alpha}},
\]
with $\eta=C^{2}C_{\alpha}$ where $C_{\alpha}$ satisfies $\rho^{2T}\le C_{\alpha}/T^{\alpha}$ for all $T\ge1$.

\section{Proof of Theorems \ref{theo63}  and \ref{theo64}} 

For convenience, we restate Theorem \ref{theo63} here:
\begin{theorem}
	Let Assumptions \ref{assumption:differentiability} (Differentiability), \ref{assumption:convexity} (Convexity), \ref{assumption:strong_convexity}(Strong convexity of $f$), \ref{assm:similarity} (Star Similarity) and  \ref{assumption:bounded_variance} (Bounded Variance at Optimum) hold. If the stepsize satisfies $\gamma \leq \frac{\mu}{4\delta_*^2}$, we have, for every $k\geq 0$,
	\begin{equation}
		\mathbb{E}\big[\|x_k - x_*\|^2\big] \leq \left(1 - \min\left(\frac{\gamma\mu}{4},\frac{1}{2}\right)\right)^{k}\|x_0-x_*\|^2 + \max\left(\frac{4}{\gamma\mu},2\right)2\gamma^2\sigma_*^2.
	\end{equation}
	\end{theorem}
	\begin{proof}
		Define
		\begin{equation*}
			\bar{x}_{k+1} = \mathbb{E}[x_{k+1}|\mathcal{F}_k] .
		\end{equation*}
		Then
		\begin{align*}
			\|x_{k+1} - x_*\|^2 &= \|x_k-x_*\|^2 - 2\langle x_{k+1}-x_*,x_k-x_{k+1}\rangle - \|x_k-x_{k+1}\|^2 \\
			&\leq \|x_k - x_*\|^2 - 2\gamma(f_{\xi_k}(x_{k+1}) - f_{\xi_k}(x_*)) - \|x_k-x_{k+1}\|^2 \\
			&\leq \|x_k - x_*\|^2 - 2\gamma(f_{\xi_k}(\bar{x}_{k+1})+\langle \nabla f_{\xi_k}(\bar{x}_{k+1}),x_{k+1}-\bar{x}_{k+1}\rangle - f_{\xi_k}(x_*)) - \|x_k - x_{k+1}\|^2.
		\end{align*}
		Taking the expectation conditioned on $\mathcal{F}_k$ and using the equality $\mathbb{E}[\|x - c\|^2] = \mathbb{E}[\|x - \mathbb{E}[x]\|^2] + \|\mathbb{E}[x]-c\|^2$, we get
		\begin{align}
			\nonumber
			\mathbb{E}\left[\|x_{k+1} - x_*\|^2| \mathcal{F}_k\right] &\leq \|x_k - x_*\|^2 - 2\gamma \mathbb{E}\left[f_{\xi_k}(\bar{x}_{k+1})+\langle \nabla f_{\xi_k}(\bar{x}_{k+1}),x_{k+1}-\bar{x}_{k+1}\rangle - f_{\xi_k}(x_*)|\mathcal{F}_k\right] \\
			&\quad- \mathbb{E}\left[\|x_k - x_{k+1}\|^2|\mathcal{F}_k\right]\notag\\
			\nonumber
			&= \|x_k - x_*\|^2 - 2\gamma\left(f(\bar{x}_{k+1}) - f(x_*)+\mathbb{E}\left[\langle \nabla f_{\xi_k}(\bar{x}_{k+1}),x_{k+1}-\bar{x}_{k+1}\rangle|\mathcal{F}_k\right]\right) \\ 
			&\quad - \mathbb{E}\left[\|x_{k+1} - \bar{x}_{k+1}\|^2|x_k\right] - \|x_k - \bar{x}_{k+1}\|^2. \label{eq:similarity_start}
		\end{align}
		Since $\mathbb{E}\left[\langle f(\bar{x}_{k+1}),x_{k+1}-\bar{x}_{k+1}\rangle|\mathcal{F}_k\right] = 0$, we can write
		\begin{align*}
			\nonumber
			\mathbb{E}\left[ \|x_{k+1} - x_*\|^2 \mid \mathcal{F}_k\right] &\leq \|x_k - x_*\|^2 - 2\gamma\left(f(\bar{x}_{k+1}) - f(x_*) + \mathbb{E}\left[\langle \nabla f_{\xi_k}(\bar{x}_{k+1}), x_{k+1} - \bar{x}_{k+1} \rangle \mid \mathcal{F}_k\right]\right) \\ 
			\nonumber
			&\quad - \mathbb{E}\left[\|x_{k+1} - \bar{x}_{k+1}\|^2 \mid \mathcal{F}_k\right] - \|x_k - \bar{x}_{k+1}\|^2\\
			\nonumber
			&= \|x_k - x_*\|^2 - 2\gamma\left(f(\bar{x}_{k+1}) - f(x_*)\right)\notag\\
			&\quad - 2\gamma\mathbb{E}\left[\langle \nabla f_{\xi_k}(\bar{x}_{k+1})-\nabla f(\bar{x}_{k+1})-\nabla f_{\xi_k}(x_{*}), x_{k+1} - \bar{x}_{k+1} \rangle \mid \mathcal{F}_k\right]\\
			\nonumber
			&\quad +2\gamma\left(\mathbb{E}\left[\langle \nabla f_{\xi_k}(x_{*}), \bar{x}_{k+1} - x_{k+1} \rangle \mid \mathcal{F}_k\right]\right)  
			- \mathbb{E}\left[\|x_{k+1} - \bar{x}_{k+1}\|^2 \mid \mathcal{F}_k\right] - \|x_k - \bar{x}_{k+1}\|^2 .
		\end{align*}
		Let $a\in (0,1)$. Then
		\begin{align*}
			\mathbb{E}\left[ \|x_{k+1} - x_*\|^2 \mid \mathcal{F}_k\right] &\leq \|x_k - x_*\|^2 - 2\gamma\left(f(\bar{x}_{k+1}) - f(x_*)\right)\\
			&\quad - \frac{2}{a}E\left[\langle \gamma(\nabla f_{\xi_k}(\bar{x}_{k+1})-\nabla f(\bar{x}_{k+1})-\nabla f_{\xi_k}(x_{*})), a(x_{k+1} - \bar{x}_{k+1} )\rangle \mid \mathcal{F}_k\right]\\
			\nonumber
			&\quad +2\gamma\left(\mathbb{E}\left[\langle \nabla f_{\xi_k}(x_{*}), \bar{x}_{k+1} - x_{k+1} \rangle \mid \mathcal{F}_k\right]\right)  
			- \mathbb{E}\left[\|x_{k+1} - \bar{x}_{k+1}\|^2 \mid \mathcal{F}_k\right] - \|x_k - \bar{x}_{k+1}\|^2 .
		\end{align*}
		Using strong convexity of $f$ and the identity  $2\langle c,b\rangle = \|c+b\|^2 - \|c\|^2 - \|b\|^2$, we obtain
		\begin{align*}
			\mathbb{E}\left[ \|x_{k+1} - x_*\|^2\right] &\leq \|x_k - x_*\|^2 - \gamma\mu \|\bar{x}_{k+1} - x_*\|^2 \\
			&\quad- \frac{1}{a}\mathbb{E}\left[\|\gamma(\nabla f_{\xi_k}(\bar{x}_{k+1})-\nabla f(\bar{x}_{k+1})-\nabla f_{\xi_k}(x_{*})) + a(x_{k+1} - \bar{x}_{k+1} )\|^2 \mid \mathcal{F}_k\right]\\
			\nonumber
			&\quad  + \frac{\gamma^2}{a}\mathbb{E}\left[\|(\nabla f_{\xi_k}(\bar{x}_{k+1})-\nabla f(\bar{x}_{k+1})-\nabla f_{\xi_k}(x_{*}))\|^2 \mid \mathcal{F}_k\right] - (1-a) \mathbb{E}\left[\|x_{k+1} - \bar{x}_{k+1} \|^2 \mid \mathcal{F}_k\right]\\
			&\quad +2\gamma\left(\mathbb{E}\left[\langle \nabla f_{\xi_k}(x_{*}), \bar{x}_{k+1} - x_{k+1} \rangle \mid \mathcal{F}_k\right]\right)  
			- \|x_k - \bar{x}_{k+1}\|^2.
		\end{align*}
		Using Assumption~\ref{assm:similarity} and Young's inequality, we get
		\begin{align*}
			\mathbb{E}\left[ \|x_{k+1} - x_*\|^2\right]&\leq \|x_k - x_*\|^2 - \gamma\mu \|\bar{x}_{k+1} - x_*\|^2 + \frac{\gamma^2 \delta_*^2}{a}\|\bar{x}_{k+1}-x_*\|^2 -  (1-a) \mathbb{E}\left[\|x_{k+1} - \bar{x}_{k+1} \|^2 \mid x_k\right] \\
			&\quad + \frac{\gamma^2}{s} \mathbb{E}\left[\|\nabla f_{\xi_k}(x_{*})\|^2\mid x_k\right] + s \mathbb{E}\left[ \|x_{k+1} - \bar{x}_{k+1}\|^2 \mid x_k\right]  
			- \|x_k - \bar{x}_{k+1}\|^2.
		\end{align*}
		By substituting $s = 1- a$ and choosing $a = \frac{1}{2}$, we get
		\begin{align*}
			\mathbb{E}\left[ \|x_{k+1} - x_*\|^2 \mathcal{F}_k\right] &\leq \|x_k - x_*\|^2 - \gamma\mu \|\bar{x}_{k+1} - x_*\|^2 + 2\gamma^2 \delta_*^2\|\bar{x}_{k+1}-x_*\|^2 + 2\gamma^2 \mathbb{E}\left[\|\nabla f_{\xi_k}(x_{*})\|^2\mid x_k\right]\\
			&\quad - \|x_k - \bar{x}_{k+1}\|^2 .
		\end{align*}		
		We require $2\gamma^2\delta_*^2\leq \frac{\gamma\mu}{2}$, which implies $\gamma \leq \frac{\mu}{4\delta_*^2}$. Therefore, we get
		\begin{align*}
			\mathbb{E}\left[ \|x_{k+1} - x_*\|^2 \mid \mathcal{F}_k\right] &\leq \|x_k - x_*\|^2 - \frac{\gamma\mu}{2}\|\bar{x}_{k+1} - x_*\|^2- \|x_k - \bar{x}_{k+1}\|^2 + 2\gamma^2 \mathbb{E}\left[\|\nabla f_{\xi_k}(x_{*})\|^2\mid x_k\right] \\
			&\leq \|x_k-x_*\|^2 - \min\left(\frac{\gamma\mu}{2},1\right)(\|\bar{x}_{k+1} - x_*\|^2- \|x_k - \bar{x}_{k+1}\|^2) + 2\gamma^2 \mathbb{E}\left[\|\nabla f_{\xi_k}(x_{*})\|^2\mid x_k\right] \\
			&\leq \left(1-\min\left(\frac{\gamma\mu}{4},\frac{1}{2}\right)\right)
			\|x_k-x_*\|^2 + 2\gamma^2\sigma_*^2.
		\end{align*}
		Taking the full expectation, we obtain
		\begin{equation*}
			\mathbb{E}[\|x_{k+1}-x_{*}\|^2] \leq \left(1-\min\left(\frac{\gamma\mu}{4},\frac{1}{2}\right)\right)\mathbb{E}[\|x_k-x_*\|^2] + 2\gamma^2\sigma_*^2.
		\end{equation*}
		By applying the inequality recursively, we derive
		\begin{align*}
			\mathbb{E}[\|x_{k+1}-x_{*}\|^2] &\leq \left(1-\min\left(\frac{\gamma\mu}{4},\frac{1}{2}\right)\right)^{k+1}\|x_0-x_*\|^2 \\
			&\quad + \left(\left(1 - \min\left(\frac{\gamma\mu}{4},\frac{1}{2}\right)\right)^0 + \dots + \left(1 - \min\left(\frac{\gamma\mu}{4},\frac{1}{2}\right)\right)^k\right)2\gamma^2\sigma_*^2 \\
			&\leq \left(1 - \min\left(\frac{\gamma\mu}{4},\frac{1}{2}\right)\right)^{k+1}\|x_0-x_*\|^2 + \max\left(\frac{4}{\gamma\mu},2\right)2\gamma^2\sigma_*^2.
		\end{align*}
	\end{proof}

For convenience, we restate Theorem \ref{theo64} here:
\begin{theorem}
	Let Assumptions \ref{assumption:differentiability} (Differentiability), \ref{assumption:convexity} (Convexity), \ref{assm:similarity} 
	(Star Similarity) and \ref{assumption:bounded_variance} (Bounded Variance) hold. Consider SPPM-inexact with Assumption~\ref{assumption:inexactness_condition} satisfied. If the stepsize satisfies $\gamma \leq \frac{\mu(1-c)}{4\delta_*^2}$ and $T$ is chosen sufficiently large such that
	$\frac{\eta \gamma^2}{T^\alpha} \leq c$, where $0<c<1$ is a constant, then
	SPPM-inexact satisfies, for every $k\geq 0$, 
	\begin{equation*}
	\mathbb{E}\big[\|x_k - x_*\|^2\big] \leq \left(1 - \frac{1}{2}\min\left(\frac{\gamma\mu}{2},1-c\right)\right)^{k}\|x_0-x_*\|^2 + \max\left(\frac{2}{\gamma\mu},\frac{1}{1-c}\right)\frac{4\gamma^2\sigma_*^2}{(1-c)}.
	\end{equation*}
\end{theorem}
\begin{proof}
	To avoid repetition, we start from \eqref{eq:sim_inexact_base}:
	\begin{align*}
		\mathbb{E}\left[\|x_{k+1} - x_*\|^2| x_k\right] &\leq \|x_k - x_*\|^2 - 2\gamma\left(f(\bar{x}_{k+1}) - f(x_*)+\mathbb{E}\left[\langle \nabla f_{\xi_k}(\bar{x}_{k+1}),x^{\Psi}_{k}-\bar{x}_{k+1}\rangle|x_k\right]\right) \\ 
		&\quad - (1-\frac{\eta\gamma^2}{T^{\alpha}})\mathbb{E}\left[\|x^{\Psi}_{k} - \bar{x}_{k+1}\|^2|x_k\right] - (1-\frac{\eta\gamma^2}{T^{\alpha}})\|x_k - \bar{x}^{k+1}\|^2 .
	\end{align*}
	Now since $\mathbb{E}\left[\langle f(\bar{x}_{k+1}),x^{\Psi}_{k}-\bar{x}_{k+1}\rangle|x_k\right] = 0$, we can write
	\begin{align*}
		\nonumber
		\mathbb{E}\left[ \|x_{k+1} - x_*\|^2 \mid x_k\right] &\leq \|x_k - x_*\|^2 - 2\gamma\left(f(\bar{x}_{k+1}) - f(x_*)\right) \\
		&\quad- 2\gamma\mathbb{E}\left[\langle \nabla f_{\xi_k}(\bar{x}_{k+1})-\nabla f(\bar{x}_{k+1})-\nabla f_{\xi_k}(x_{*}), x^{\Psi}_{k} - \bar{x}_{k+1} \rangle \mid x_k\right]\\
		\nonumber
		&\quad +2\gamma\left(\mathbb{E}\left[\langle \nabla f_{\xi_k}(x_{*}), \bar{x}_{k+1} - x^{\Psi}_{k} \rangle \mid x_k\right]\right)  
		- (1-\frac{\eta\gamma^2}{T^{\alpha}})\mathbb{E}\left[\|x^{\Psi}_{k} - \bar{x}_{k+1}\|^2 \mid x_k\right] \\
		&\quad- \left(1-\frac{\eta\gamma^2}{T^{\alpha}}\right)\|x_k - \bar{x}_{k+1}\|^2 .
	\end{align*}
	Let $a\in (0,1)$. Then
	\begin{align*}
		\nonumber
		\mathbb{E}\left[ \|x_{k+1} - x_*\|^2 \mid x_k\right] &\leq \|x_k - x_*\|^2 - 2\gamma\left(f(\bar{x}_{k+1}) - f(x_*)\right)\\ 
		\nonumber
		&\quad- 2\frac{1}{a}\left(\mathbb{E}\left[\langle \gamma(\nabla f_{\xi_k}(\bar{x}_{k+1})-\nabla f(\bar{x}_{k+1})-\nabla f_{\xi_k}(x_{*})),a( x^{\Psi}_{k} - \bar{x}_{k+1} ) \rangle \mid x_k\right]\right)\\
		\nonumber
		&\quad +2\gamma\left(\mathbb{E}\left[\langle \nabla f_{\xi_k}(x_{*}), \bar{x}_{k+1} - x^{\Psi}_{k} \rangle \mid x_k\right]\right)  
		- (1-\frac{\eta\gamma^2}{T^{\alpha}})\mathbb{E}\left[\|x^{\Psi}_{k} - \bar{x}_{k+1}\|^2 \mid x_k\right] \\
		&\quad- \left(1-\frac{\eta\gamma^2}{T^{\alpha}}\right)\|x_k - \bar{x}_{k+1}\|^2.
	\end{align*}
	Using strong convexity of $f$ and the identity $2\langle c,b\rangle = \|c+b\|^2 - \|c\|^2 - \|b\|^2$, we obtain
	\begin{align*}
		\mathbb{E}\left[ \|x_{k+1} - x_*\|^2 \mid x_k\right] &\leq \|x_k - x_*\|^2 - \gamma\mu \|\bar{x}_{k+1} - x_*\|^2 \\
		&\quad - \frac{1}{a}\left(\mathbb{E}\left[\|\gamma(\nabla f_{\xi_k}(\bar{x}_{k+1})-\nabla f(\bar{x}_{k+1})-\nabla f_{\xi_k}(x_{*})) + a(x^{\Psi}_{k} - \bar{x}_{k+1} )\|^2 \mid x_k\right]\right)\\
		\nonumber
		&\quad  + \frac{\gamma^2}{a}\mathbb{E}\left[\|(\nabla f_{\xi_k}(\bar{x}_{k+1})-\nabla f(\bar{x}_{k+1})-\nabla f_{\xi_k}(x_{*}))\|^2 \mid x_k\right] + a  \mathbb{E}\left[\|x^{\Psi}_{k} - \bar{x}_{k+1} \|^2 \mid x_k\right]\\
		&\quad +2\gamma\left(\mathbb{E}\left[\langle \nabla f_{\xi_k}(x_{*}), \bar{x}_{k+1} - x^{\Psi}_{k} \rangle \mid x_k\right]\right)  
		- (1-\frac{\eta\gamma^2}{T^{\alpha}})\mathbb{E}\left[\|x^{\Psi}_{k} - \bar{x}_{k+1}\|^2 \mid x_k\right] \\
		&\quad- \left(1-\frac{\eta\gamma^2}{T^{\alpha}}\right)\|x_k - \bar{x}_{k+1}\|^2 .
	\end{align*}
	Using Assumption \ref{assm:similarity} and Young's inequality, we get
	\begin{align*}
		\mathbb{E}\left[ \|x_{k+1} - x_*\|^2 \mid x_k\right] &\leq \|x_k - x_*\|^2 - \gamma\mu \|\bar{x}_{k+1} - x_*\|^2  + \frac{\gamma^2\delta_*^2}{a}\|\bar{x}_{k+1}-x_*\|^2 + a  \mathbb{E}\left[\|x^{\Psi}_{k} - \bar{x}_{k+1} \|^2 \mid x_k\right]\\
		&\quad +\frac{\gamma^2}{s}\mathbb{E}\left[\|\nabla f_{\xi_k}(x_{*})\|^2 \mid x_k\right] + sE\left[\|\bar{x}_{k+1} - x^{\Psi}_{k}\|^2 \mid x_k\right]  \\
		&\quad- \left(1-\frac{\eta\gamma^2}{T^{\alpha}}\right)\mathbb{E}\left[\|x^{\Psi}_{k} - \bar{x}_{k+1}\|^2 \mid x_k\right]- \left(1-\frac{\eta\gamma^2}{T^{\alpha}}\right)\|x_k - \bar{x}_{k+1}\|^2 .
	\end{align*}
	Choosing $s = a = \frac{1-\frac{\eta\gamma^2}{T^{\alpha}}}{2}$, we derive
	\begin{align*}
		\mathbb{E}\left[ \|x_{k+1} - x_*\|^2 \mid x_k\right] &\leq \|x_k - x_*\|^2 - \gamma\mu \|\bar{x}_{k+1} - x_*\|^2  + \frac{2\gamma^2\delta_*^2}{1-\frac{\eta\gamma^2}{T^{\alpha}}}\|\bar{x}_{k+1}-x_*\|^2 \\
		&\quad+\frac{2\gamma^2}{1-\frac{\eta\gamma^2}{T^{\alpha}}}\mathbb{E}\left[\|\nabla f_{\xi_k}(x_{*})\|^2 \mid x_k \right] - \left(1-\frac{\eta\gamma^2}{T^{\alpha}}\right)\|x_k - \bar{x}_{k+1}\|^2 .
	\end{align*}
	Let \(T\) be large enough such that \(\frac{\eta\gamma^2}{T^{\alpha}} \leq c<1\). Then
	\begin{align*}
		\mathbb{E}\left[ \|x_{k+1} - x_*\|^2 \mid x_k\right] &\leq \|x_k - x_*\|^2 - \gamma\mu \|\bar{x}_{k+1} - x_*\|^2  + \frac{2\gamma^2\delta_*^2}{1-c}\|\bar{x}_{k+1}-x_*\|^2 + \frac{2\gamma^2}{1-c} \mathbb{E}\left[\|\nabla f_{\xi_k}(x_{*})\|^2\mid x_k \right] \\
		&\quad- (1-c)\|x_k - \bar{x}_{k+1}\|^2 .
	\end{align*}
	We want $2\gamma^2\delta_*^2\leq \frac{\gamma\mu(1-c)}{2}$, which implies $\gamma \leq \frac{\mu(1-c)}{4\delta_*^2}$, so we get
	\begin{align*}
		\mathbb{E}\left[ \|x_{k+1} - x_*\|^2 \mid x_k\right] &\leq \|x_k - x_*\|^2 - \frac{\gamma\mu}{2} \|\bar{x}_{k+1} - x_*\|^2  - (1-c)\|x_k - \bar{x}_{k+1}\|^2 + \frac{2\gamma^2}{1-c} \mathbb{E}\left[\|\nabla f_{\xi_k}(x_{*})\|^2 \mid x_k\right] \\
		&\leq  \|x_k - x_*\|^2 - \frac{1}{2}\min\left(\frac{\gamma\mu}{2},1-c\right) \|x_k - \bar{x}_{*}\|^2 + \frac{2\gamma^2}{1-c} \mathbb{E}\left[\|\nabla f_{\xi_k}(x_{*})\|^2 \mid x_k\right] \\
		&\overset{\ref{assumption:bounded_variance}}{\leq} \left(1 - \frac{1}{2}\min\left(\frac{\gamma\mu}{2},1-c\right)\right)\|x_k - x_*\|^2 + \frac{2\gamma^2\sigma_*^2}{1-c}.
	\end{align*}
	Taking the full expectation, we obtain
	\begin{equation*}
		\mathbb{E}\left[ \|x_{k+1} - x_*\|^2 \right] \leq \left(1 - \frac{1}{2}\min\left(\frac{\gamma\mu}{2},1-c\right)\right)\mathbb{E}[\|x_k - x_*\|^2] + \frac{2\gamma^2\sigma_*^2}{1-c}.
	\end{equation*}
	By applying the inequality recursively, we derive
	\begin{align*}
		\mathbb{E}[\|x_{k+1}-x_{*}\|^2] &\leq \left(1 - \frac{1}{2}\min\left(\frac{\gamma\mu}{2},1-c\right)\right)^{k+1}\|x_0-x_*\|^2 \\
		&\quad + \left(\left(1 - \frac{1}{2}\min\left(\frac{\gamma\mu}{2},1-c\right)\right)^0 + \dots + \left(1 - \frac{1}{2}\min\left(\frac{\gamma\mu}{2},1-c\right)\right)^k\right)\frac{2\gamma^2\sigma_*^2}{1-c} \\
		&\leq \left(1 - \frac{1}{2}\min\left(\frac{\gamma\mu}{2},1-c\right)\right)^{k+1}\|x_0-x_*\|^2 + \max\left(\frac{2}{\gamma\mu},\frac{1}{1-c}\right)\frac{4\gamma^2\sigma_*^2}{(1-c)}.
	\end{align*}
\end{proof}

\section{Convergence under Interpolation Regime and Expected Similarity}\label{app:interpolation}

\begin{theorem}
	Let Assumptions \ref{assumption:differentiability} (Differentiability), \ref{assumption:convexity} (Convexity), \ref{assumption:strong_convexity}(Strong convexity of $f$), \ref{assumption:interpolation} (Interpolation), and \ref{assm:similarity} (Star Similarity) hold. If the stepsize satisfies $\gamma \leq \frac{\mu}{2\delta_*^2}$, then we have, for every $k\geq 0$,
	\begin{equation}
		\mathbb{E}\big[\|x_k - x_*\|^2\big] \leq \left(1-\min\left(\frac{\gamma\mu}{4},\frac{1}{2}\right)\right)^{k}\|x_{0}-x_*\|^2.		
	\end{equation}
	\begin{proof}
		Define
		\begin{equation}
			\bar{x}_{k+1} = \mathbb{E}[x_{k+1}|\mathcal{F}_k]. \label{def:conditional_point}
		\end{equation}
		Then
		\begin{align*}
			\|x_{k+1} - x_*\|^2 &= \|x_k-x_*\|^2 - 2\langle x_{k+1}-x_*,x_k-x_{k+1}\rangle - \|x_k-x_{k+1}\|^2 \\
			&\overset{\ref{eq:convexity_interpolation_ineq}}{\leq} \|x_k - x_*\|^2 - 2\gamma(f_{\xi_k}(x_{k+1}) - f_{\xi_k}(x_*)) - \|x_k-x_{k+1}\|^2 \\
			&\overset{\ref{assumption:convexity}}{\leq} \|x_k - x_*\|^2 - 2\gamma(f_{\xi_k}(\bar{x}_{k+1})+\langle \nabla f_{\xi_k}(\bar{x}_{k+1}),x_{k+1}-\bar{x}_{k+1}\rangle - f_{\xi_k}(x_*)) - \|x_k - x_{k+1}\|^2.
		\end{align*}
		Taking the expectation conditioned on $\mathcal{F}_k$ and using the identity $\mathbb{E}[\|x - c\|^2] = E[\|x - \mathbb{E}[x]\|^2] + \|\mathbb{E}[x]-c\|^2$, we obtain
		\begin{align}
			\nonumber
			\mathbb{E}\left[\|x_{k+1} - x_*\|^2| \mathcal{F}_k\right] &\leq \|x_k - x_*\|^2 - 2\gamma \mathbb{E}\left[f_{\xi_k}(\bar{x}_{k+1})+\langle \nabla f_{\xi_k}(\bar{x}_{k+1}),x_{k+1}-\bar{x}_{k+1}\rangle - f_{\xi_k}(x_*)|\mathcal{F}_k\right] \\
			\nonumber
			&\quad - \mathbb{E}\left[\|x_k - x_{k+1}\|^2|\mathcal{F}_k\right]\\
			\nonumber
			&= \|x_k - x_*\|^2 - 2\gamma\left(f(\bar{x}_{k+1}) - f(x_*)+\mathbb{E}\left[\langle \nabla f_{\xi_k}(\bar{x}_{k+1}),x_{k+1}-\bar{x}_{k+1}\rangle|\mathcal{F}_k\right]\right) \\ 
			&\quad - \mathbb{E}\left[\|x_{k+1} - \bar{x}_{k+1}\|^2|\mathcal{F}_k\right] - \|x_k - \bar{x}_{k+1}\|^2. \label{similarity}
		\end{align}
		Since $\mathbb{E}\left[\langle f(\bar{x}_{k+1}),x_{k+1}-\bar{x}_{k+1}\rangle|\mathcal{F}_k\right] = 0$, we can write
		\begin{align*}
			\mathbb{E}\left[\|x_{k+1} - x_*\|^2| x_k\right] &= \|x_k - x_*\|^2- 2\gamma\left(f(\bar{x}_{k+1}) - f(x_*)\right.\\
			&\quad+\mathbb{E}\left[\langle \nabla f_{\xi_k}(\bar{x}_{k+1}) - \nabla f(\bar{x}_{k+1}),x_{k+1}-\bar{x}_{k+1}\rangle|x_k\right]\big) \\ 
			\nonumber
			&\quad- \mathbb{E}\left[\|x_{k+1} - \bar{x}_{k+1}\|^2|x_k\right] - \|x_k - \bar{x}_{k+1}\|^2. \notag
		\end{align*}
		Using strong convexity of $f$ and the identity $2\langle a,b\rangle = \|a+b\|^2 - \|a\|^2 - \|b\|^2$, we derive
		\begin{align*}
			\mathbb{E}\left[\|x_{k+1} - x_*\|^2| \mathcal{F}_k\right] &\leq \|x_k - x_*\|^2 -\gamma\mu\|\bar{x}_{k+1}-x_*\|^2 - \mathbb{E}\left[\|x_{k+1}- \bar{x}_{k+1}\|^2|\mathcal{F}_k\right] - \|x_k - \bar{x}_{k+1}\|^2 \\ 
			\nonumber
			&\quad - \mathbb{E}\Big[\|\gamma(\nabla f_{\xi_k}(\bar{x}_{k+1}) - \nabla f(\bar{x}_{k+1})) + (x_{k+1}-\bar{x}_{k+1})\|^2 \\
			&\quad \quad \quad- \gamma^2\|\nabla f_{\xi_k}(\bar{x}_{k+1}) - \nabla f(\bar{x}_{k+1})\|^2 - \|\bar{x}_{k+1}-x_{k+1}\|^2 \mid \mathcal{F}_k \Big] \\
			&\leq \|x_k - x_*\|^2 -\gamma\mu\|\bar{x}_{k+1}-x_*\|^2 - \|x_k - \bar{x}_{k+1}\|^2 \\
			&\quad+ \gamma^2 \mathbb{E}\left[\|\nabla f_{\xi_k}(\bar{x}_{k+1}) - \nabla f(\bar{x}_{k+1})\|^2|\mathcal{F}_k\right].
		\end{align*}
		Finally, using the star similarity condition, we obtain
		\begin{align*}
			\mathbb{E}\left[\|x_{k+1} - x_*\|^2| \mathcal{F}_k\right] &{\leq} \|x_k - x_*\|^2 -\gamma\mu\|\bar{x}_{k+1}-x_*\|^2 - \|x_k - \bar{x}_{k+1}\|^2 + \gamma^2\delta_*^2 \|\bar{x}_{k+1} - x_{*}\|^2\label{eq:similarity_step} \\
			&= \|x_k - x_*\|^2 -(\gamma\mu - \gamma^2\delta_*^2)\|\bar{x}_{k+1}-x_*\|^2 - \|x_k - \bar{x}_{k+1}\|^2  .
		\end{align*}
		If $\gamma^2\delta_*^2 \leq\frac{1}{2} \gamma \mu $, then $\gamma \leq \frac{\mu}{2\delta_*^2}$. Under this condition, we have
		\begin{align*}
			\nonumber
			\mathbb{E}\left[\|x_{k+1} - x_*\|^2| \mathcal{F}_k\right] &\leq \|x_k - x_*\|^2 -\frac{\gamma\mu}{2}\|\bar{x}_{k+1}-x_*\|^2 - \|x_k - \bar{x}_{k+1}\|^2 \\
			\nonumber
			&\leq \|x_k - x_*\|^2 -\min\left(\frac{\gamma\mu}{2},1\right)(\|\bar{x}_{k+1}-x_*\|^2 - \|x_k - \bar{x}_{k+1}\|^2) \\
			&\leq  \|x_k - x_*\|^2 -\min\left(\frac{\gamma\mu}{4},\frac{1}{2}\right)\|x_{k}-x_*\|^2 = \left(1-\min\left(\frac{\gamma\mu}{4},\frac{1}{2}\right)\right)\|x_{k}-x_*\|^2.
		\end{align*}
		Taking the full expectation, we obtain
		\begin{equation*}
			\mathbb{E}\left[\|x_{k+1} - x_*\|^2\right] \leq \left(1-\min\left(\frac{\gamma\mu}{4},\frac{1}{2}\right)\right)\mathbb{E}\left[\|x_{k}-x_*\|^2\right].
		\end{equation*}
		Applying this inequality recursively, we get
		\begin{equation*}
			\mathbb{E}\left[\|x_{k+1} - x_*\|^2\right] \leq \left(1-\min\left(\frac{\gamma\mu}{4},\frac{1}{2}\right)\right)^{k+1}\|x_{0}-x_*\|^2.
		\end{equation*}
	\end{proof}
\end{theorem}

\begin{theorem}
	Let Assumptions \ref{assumption:differentiability} (Differentiability), \ref{assumption:convexity} (Convexity), \ref{assumption:interpolation} (Interpolation) and \ref{assm:similarity} 
	(Star Similarity) hold. Consider SPPM-inexact with Assumption \ref{assumption:inexactness_condition} satsified. If the stepsize satisfies $\gamma \leq \frac{\mu(1-c)}{2\delta_*^2}$ and $T$ is chosen sufficiently large such that
	$\frac{\eta \gamma^2}{T^\alpha} \leq c < 1$, then the iterates of SPPM-inexact
	satisfy, for every $k\geq 0$,
	\begin{equation*}
		\mathbb{E}\left[\|x_{k} - x_*\|^2\right] \leq \left(1 - \frac{1}{2}\min\left(\frac{\gamma\mu}{2},1-c\right)\right)^{k}\|x_{0} - x_*\|^2.
	\end{equation*}
\end{theorem}
\begin{proof}
	To avoid repetitions, we start from \eqref{eq:inexactness_main_inequality}:
	\begin{equation*}
		\|x_{k+1}-x_*\| \leq  \|x_k - x_*\|^2 - 2\gamma(f_{\xi_k}(x_{k}^{\Psi}) - f_{\xi_k}(x_*)) -(1-\frac{\eta\gamma^2}{T^{\alpha}})\|x_k-x_k^{\Psi}\|^2.
	\end{equation*}
	Instead of \eqref{def:conditional_point}, we define $\bar{x}_{k+1} = \mathbb{E}[x^{\Psi}_k|x_k]$.
	
	Using convexity of $f_{\xi_k}$, we get
	\begin{align*}
		\|x_{k+1}-x_*\| &\leq \|x_k - x_*\|^2 - 2\gamma(f_{\xi_k}(\bar{x}_{k+1})+\langle \nabla f_{\xi_k}(\bar{x}_{k+1}),x^{\Psi}_{k}-\bar{x}_{k+1}\rangle - f_{\xi_k}(x_*)) \\
		&\quad - (1-\frac{\eta\gamma^2}{T^{\alpha}}) \|x_k - x^{\Psi}_k\|^2.
	\end{align*}
	Taking the expectation conditioned on $x_k$ and using the equality $\mathbb{E}[\|x - c\|^2] = \mathbb{E}[\|x - E[x]\|^2] + \|\mathbb{E}[x]-c\|^2$, we obtain
	\begin{align}
		\nonumber
		\mathbb{E}\left[\|x_{k+1} - x_*\|^2| x_k\right] &\leq \|x_k - x_*\|^2  - (1-\frac{\eta\gamma^2}{T^{\alpha}})E\left[\|x_k -x^{\Psi}_{k}\|^2|x_k \right]\\
		\nonumber
		&\quad - 2\gamma \mathbb{E}\left[f_{\xi_k}(\bar{x}_{k+1})+\langle \nabla f_{\xi_k}(\bar{x}_{k+1}),x^{\Psi}_{k}-\bar{x}_{k+1}\rangle - f_{\xi_k}(x_*)|x_k\right]\\
		\nonumber
		&= \|x_k - x_*\|^2 - 2\gamma\left(f(\bar{x}_{k+1}) - f(x_*)+\mathbb{E}\left[\langle \nabla f_{\xi_k}(\bar{x}_{k+1}),x^{\Psi}_{k}-\bar{x}_{k+1}\rangle|x_k\right]\right) \\ 
		&\quad - (1-\frac{\eta\gamma^2}{T^{\alpha}})\mathbb{E}\left[\|x^{\Psi}_{k} - \bar{x}_{k+1}\|^2|x_k\right] - (1-\frac{\eta\gamma^2}{T^{\alpha}})\|x_k - \bar{x}^{k+1}\|^2. \label{eq:sim_inexact_base}
	\end{align}
	Since $\mathbb{E}\left[\langle f(\bar{x}_{k+1}),x^{\Psi}_{k}-\bar{x}_{k+1}\rangle|x_k\right] = 0$, we can write
	\begin{align}
		\nonumber
		\mathbb{E}\left[\|x_{k+1} - x_*\|^2| x_k\right] &\leq \|x_k - x_*\|^2 - 2\gamma\left(f(\bar{x}_{k+1}) - f(x_*)+\mathbb{E}\left[\langle \nabla f_{\xi_k}(\bar{x}_{k+1}) - f(\bar{x}_{k+1}),x^{\Psi}_{k}-\bar{x}_{k+1}\rangle|x_k\right]\right) \\ 
		&\quad - (1-\frac{\eta\gamma^2}{T^{\alpha}})\mathbb{E}\left[\|x^{\Psi}_{k} - \bar{x}_{k+1}\|^2|x_k\right] - (1-\frac{\eta\gamma^2}{T^{\alpha}})\|x_k - \bar{x}^{k+1}\|^2.
	\end{align}
	Let $a \in (0,1)$. Then
	\begin{align*}
		\nonumber
		\mathbb{E}\left[\|x_{k+1} - x_*\|^2| x_k\right] &\leq \|x_k - x_*\|^2 - 2\gamma\left(f(\bar{x}_{k+1}) - f(x_*)\right) \\
		&\quad - 2 \frac{1}{a}\mathbb{E}\left[\langle \gamma(\nabla f_{\xi_k}(\bar{x}_{k+1}) - f(\bar{x}_{k+1})),a(x^{\Psi}_{k}-\bar{x}_{k+1})\rangle|x_k\right] \\ 
		&\quad - (1-\frac{\eta\gamma^2}{T^{\alpha}})\mathbb{E}\left[\|x^{\Psi}_{k} - \bar{x}_{k+1}\|^2|x_k\right] - (1-\frac{\eta\gamma^2}{T^{\alpha}})\|x_k - \bar{x}^{k+1}\|^2 .
	\end{align*}
	By applying strong convexity of $f$ and using the identity $2\langle c,b\rangle = \|c+b\|^2 - \|c\|^2 - \|b\|^2$, we have
	\begin{align*}
		\mathbb{E}\left[\|x_{k+1} - x_*\|^2| x_k\right] &\leq \|x_k - x_*\|^2 -\gamma\mu\|\bar{x}_{k+1}-x_*\|^2 - (1-\frac{\eta\gamma^2}{T^{\alpha}})\mathbb{E}\left[\|x^{\Psi}_{k}- \bar{x}_{k+1}\|^2|x_k\right] \\
		&\quad- (1-\frac{\eta\gamma^2}{T^{\alpha}})\|x_k - \bar{x}_{k+1}\|^2 \\ 
		&\quad - \frac{1}{a}\mathbb{E}\left[\|\gamma(\nabla f_{\xi_k}(\bar{x}_{k+1}) - \nabla f(\bar{x}_{k+1})) + (x^{\Psi}_{k}-\bar{x}_{k+1})\|^2 |x_k\right]\\
		&\quad + \frac{\gamma^2}{a} \mathbb{E}\left[\|\nabla f_{\xi_k}(\bar{x}_{k+1}) - \nabla f(\bar{x}_{k+1})\|^2|x_k\right] + a \mathbb{E}\left[\|\bar{x}_{k+1}-x^{\Psi}_{k}\|^2|x_k\right].
	\end{align*}
	By setting $ a = (1-\frac{\eta\gamma^2}{T^{\alpha}})$, we obtain
	\begin{align*}
		\mathbb{E}\left[\|x_{k+1} - x_*\|^2| x_k\right] &\leq \|x_k - x_*\|^2 -\gamma\mu\|\bar{x}_{k+1}-x_*\|^2 - (1-\frac{\eta\gamma^2}{T^{\alpha}})\|x_k - \bar{x}_{k+1}\|^2 \\ 
		\nonumber
		&\quad + \frac{\gamma^2}{ (1-\frac{\eta\gamma^2}{T^{\alpha}})} \mathbb{E}\left[\|\nabla f_{\xi_k}(\bar{x}_{k+1}) - \nabla f(\bar{x}_{k+1})\|^2|x_k\right] .
	\end{align*}
	Now, using the star similarity assumption \eqref{assm:similarity}, we get
	\begin{align*}
		\mathbb{E}\left[\|x_{k+1} - x_*\|^2| x_k\right] &\leq \|x_k - x_*\|^2 -\gamma\mu\|\bar{x}_{k+1}-x_*\|^2 - (1-\frac{\eta\gamma^2}{T^{\alpha}})\|x_k - \bar{x}_{k+1}\|^2  \\&\quad+ \frac{\gamma^2\delta_*^2}{ (1-\frac{\eta\gamma^2}{T^{\alpha}})} \|\bar{x}_{k+1} - x_*\|^2.
	\end{align*}
	Assuming $T$ sufficiently large such that $\frac{\eta\gamma^2}{T^{\alpha}}\leq c<1$, we obtain
	\begin{align*}
		\mathbb{E}\left[\|x_{k+1} - x_*\|^2| x_k\right] &\leq \|x_k - x_*\|^2 -\gamma\mu\|\bar{x}_{k+1}-x_*\|^2 - (1-c)\|x_k - \bar{x}_{k+1}\|^2  \\&\quad+ \frac{\gamma^2\delta_*^2}{1-c} \|\bar{x}_{k+1} - x_*\|^2.
	\end{align*}
	We want $\frac{\gamma^2\delta_*^2}{1-c}\leq \frac{\gamma\mu}{2}$, which means $\gamma \leq \frac{\mu(1-c)}{2\delta_*^2}$, so we get
	\begin{align*}
		\mathbb{E}\left[\|x_{k+1} - x_*\|^2| x_k\right] &\leq \|x_k - x_*\|^2 -\min\left(\frac{\gamma\mu}{2},1-c\right)(\|\bar{x}_{k+1}-x_*\|^2 - \|x_k - \bar{x}_{k+1}\|^2) \\ 
		&\leq \|x_k-x_*\|^2 - \frac{1}{2}\min\left(\frac{\gamma\mu}{2},1-c\right)\|x_{k} - x_*\|^2\\
		&= \left(1- \frac{1}{2}\min\left(\frac{\gamma\mu}{2},1-c\right)\right)\|x_k - x_*\|^2 .
	\end{align*}
	Taking the full expectation, we obtain
	\begin{equation*}
		\mathbb{E}\left[\|x_{k+1} - x_*\|^2\right] \leq \left(1- \frac{1}{2}\min\left(\frac{\gamma\mu}{2},1-c\right)\right)\mathbb{E}\left[\|x_k - x_*\|^2\right].
	\end{equation*}
	Iterating over 
	$k$, the result follows:
	\begin{equation*}
		\mathbb{E}\big[\|x_{k+1} - x_*\|^2\big] \leq \left(1 - \frac{1}{2}\min\left(\frac{\gamma\mu}{2},1-c\right)\right)^{k+1}\|x_{0} - x_*\|^2.
	\end{equation*}
\end{proof}

\end{document}